\documentclass{amsart} 

\usepackage{amssymb,amsmath,amsfonts,amsthm,cite,mathrsfs}
\usepackage{wasysym, harmony} 
\usepackage{marvosym} 
\usepackage{enumitem}   
\usepackage{verbatim}
\usepackage{xcolor}

\usepackage{dsfont} 
\usepackage{graphicx}
\usepackage{float} 
\usepackage{listings} 
\usepackage[font=normalsize, labelfont=bf]{subcaption} 
\usepackage[font=small, labelfont=bf]{caption}
\usepackage{tikz} 
\usepackage{color,soul} 
\usepackage{CJKutf8} 
\usetikzlibrary{arrows,shapes,positioning}
\usepackage{rotating} 
\usetikzlibrary{decorations.markings}
\tikzstyle arrowstyle=[scale=1]
\tikzstyle directed=[postaction={decorate,decoration={markings,
    mark=at position .65 with {\arrow[arrowstyle]{stealth}}}}]
\tikzstyle reverse directed=[postaction={decorate,decoration={markings,
    mark=at position .65 with {\arrowreversed[arrowstyle]{stealth};}}}]

\usepackage[pdftex,colorlinks]{hyperref}

\usepackage{mathtools}

\usepackage{letltxmacro}
\LetLtxMacro\orgvdots\vdots
\LetLtxMacro\orgddots\ddots

\makeatletter
\DeclareRobustCommand\vdots{%
  \mathpalette\@vdots{}%
}
\newcommand*{\@vdots}[2]{%
  \sbox0{$#1\cdotp\cdotp\cdotp\m@th$}%
  \sbox2{$#1.\m@th$}%
  \vbox{%
    \dimen@=\wd0 %
    \advance\dimen@ -3\ht2 %
    \kern.5\dimen@
    \dimen@=\wd2 %
    \advance\dimen@ -\ht2 %
    \dimen2=\wd0 %
    \advance\dimen2 -\dimen@
    \vbox to \dimen2{%
      \offinterlineskip
      \copy2 \vfill\copy2 \vfill\copy2 %
    }%
  }%
}
\DeclareRobustCommand\ddots{%
  \mathinner{%
    \mathpalette\@ddots{}%
    \mkern\thinmuskip
  }%
}
\newcommand*{\@ddots}[2]{%
  \sbox0{$#1\cdotp\cdotp\cdotp\m@th$}%
  \sbox2{$#1.\m@th$}%
  \vbox{%
    \dimen@=\wd0 %
    \advance\dimen@ -3\ht2 %
    \kern.5\dimen@
    \dimen@=\wd2 %
    \advance\dimen@ -\ht2 %
    \dimen2=\wd0 %
    \advance\dimen2 -\dimen@
    \vbox to \dimen2{%
      \offinterlineskip
      \hbox{$#1\mathpunct{.}\m@th$}%
      \vfill
      \hbox{$#1\mathpunct{\kern\wd2}\mathpunct{.}\m@th$}%
      \vfill
      \hbox{$#1\mathpunct{\kern\wd2}\mathpunct{\kern\wd2}\mathpunct{.}\m@th$}%
    }%
  }%
}
\makeatother

\usepackage{makecell}                               
\usepackage{tabularx}                               

\usepackage{amsthm}
\newtheorem{theorem}{Theorem}[section]
\newtheorem{lemma}[theorem]{Lemma}

\newtheorem*{lemma*}{Lemma}

\newtheorem*{theorem*}{Theorem}

\theoremstyle{definition}
\newtheorem{defn}[theorem]{Definition}

\newtheorem{prop}[theorem]{Proposition}
\newtheorem{cor}[theorem]{Corollary}
\newtheorem{question}[theorem]{Question}

\numberwithin{equation}{section}
\newcommand{\sgn}{\operatorname{sgn}}
\definecolor{pink}{rgb}{1,0,1}

\newcommand{\PSL}{\operatorname{PSL}}
\newcommand{\RR}{\mathbb{R}}

\newcommand{\NN}{\mathbb{N}}
\newcommand{\ZZ}{\mathbb{Z}}

\newcommand{\R}{\mathbb{R}}
\newcommand{\C}{\mathbb{C}}
\newcommand{\N}{\mathbb{N}}
\newcommand{\Z}{\mathbb{Z}}
\newcommand{\Q}{\mathbb{Q}}
\newcommand{\T}{\mathbb{T}}
\newcommand{\GL}{\mathrm{GL}}
\newcommand{\SL}{\mathrm{SL}}
\newcommand{\G}{\Gamma}
\newcommand{\Lam}{\Lambda}
\newcommand{\g}{\gamma}

\newcommand{\pos}{\textnormal{pos}}

\DeclareMathDelimiter{\backslash}    
   {\mathord}{symbols}{"6E}{largesymbols}{"0F}
\newcommand{\cM}{\mathcal{M}}

\DeclareMathSymbol{\setminus}{\mathbin}{symbols}{"6E}

\DeclareMathOperator{\MIN}{MIN}

\DeclareMathOperator{\Span}{Span}
\DeclareMathOperator{\vol}{vol}
\DeclareMathOperator{\Pc}{P}

\DeclareMathOperator{\Ima}{Im}

\newcommand{\oO}{\mathrm{O}}
\newcommand{\pa}{\partial} 
\newcommand{\cL}{\mathcal L} 

\newcommand{\choir}{\mathscr{C}}
\newcommand{\duets}{\mathscr{D}}
\newcommand{\solos}{\mathscr{S}}

\newcommand{\csec}{\mathcal{S}}

\newcommand{\audit}{\mathscr{R}}

\newcommand\inv[1]{#1^{\text{-} 1}}

\begin{document}
\title{The isospectral problem for flat tori from three perspectives}

\author{Erik Nilsson}
\address{Department of Mathematical Sciences\\KTH Royal Institute of Technology\\
SE-10044, Stockholm}

\email{erikni6@kth.se}

\author{Julie Rowlett}
\address{Department of Mathematical Sciences \\Chalmers University of Technology and The University of Gothenburg\\
SE-41296, Gothenburg}

\email{julie.rowlett@chalmers.se}
\thanks{JR is supported by Swedish Research Council Grant GAAME 2018-03873}

\author{Felix Rydell}
\address{Department of Mathematical Sciences\\KTH Royal Institute of Technology\\
SE-10044, Stockholm}

\email{felixry@kth.se}
\thanks{FR is partially supported by the Knut and Alice Wallenberg Foundation within their WASP (Wallenberg AI, Autonomous Systems and Software Program) AI/Math initiative.}


\date{}

\dedicatory{}

\begin{abstract} Flat tori are among the only types of Riemannian manifolds for which the Laplace eigenvalues can be explicitly computed.  
In 1964, Milnor used a construction of Witt to find an example of isospectral non-isometric Riemannian manifolds, a striking and concise result that occupied one page in the Proceedings of the National Academy of Science of the USA.  Milnor's example is a pair of 16-dimensional flat tori, whose set of Laplace eigenvalues are identical, in spite of the fact that these tori are not isometric.  A natural question is:  what is the \em lowest \em dimension in which such isospectral non-isometric pairs exist?  This isospectral question for flat tori can be equivalently formulated in analytic, geometric, and number theoretic language.  We explore this question from all three perspectives and describe its resolution by Schiemann in the 1990s. Moreover, we share a number of open problems.  
\end{abstract}

\maketitle

\section{Introduction} \label{s:intro}   
The Laplace eigenvalue problem is broadly appealing because it connects physics, number theory, analysis, and geometry.  At the same time, it is a challenging and frustrating problem because in general, one cannot solve it analytically.  Wielding heavy tools from functional analysis, one can prove that solutions exist, but this is not nearly as satisfying as being able to write down a solution in closed form. There is, however, a notable exception:  flat tori.  Although there is no smooth isometric embedding of a flat $n$-dimensional torus into $n+1$ dimensional Euclidean space, there is a $C^1$ embedding discovered by Nash \cite{nash_ii} and Kuiper \cite{kuiper}.  This embedding remained mysterious, eluding visualization until 2012  by Borrelli, Jabrane, Lazarus \& Thibert \cite{borrelli_book, borrelli_pnas}.  It is apparent from this \href{http://hevea-project.fr/ENPageToreImages.html#textLicence}{visualization} in Figure \ref{fig:cupcake} that the embedding is not smooth, because  the surface of the torus exhibits a fractal behavior in the normal direction.

\begin{figure}
\centering
 \includegraphics[width=0.8\linewidth]{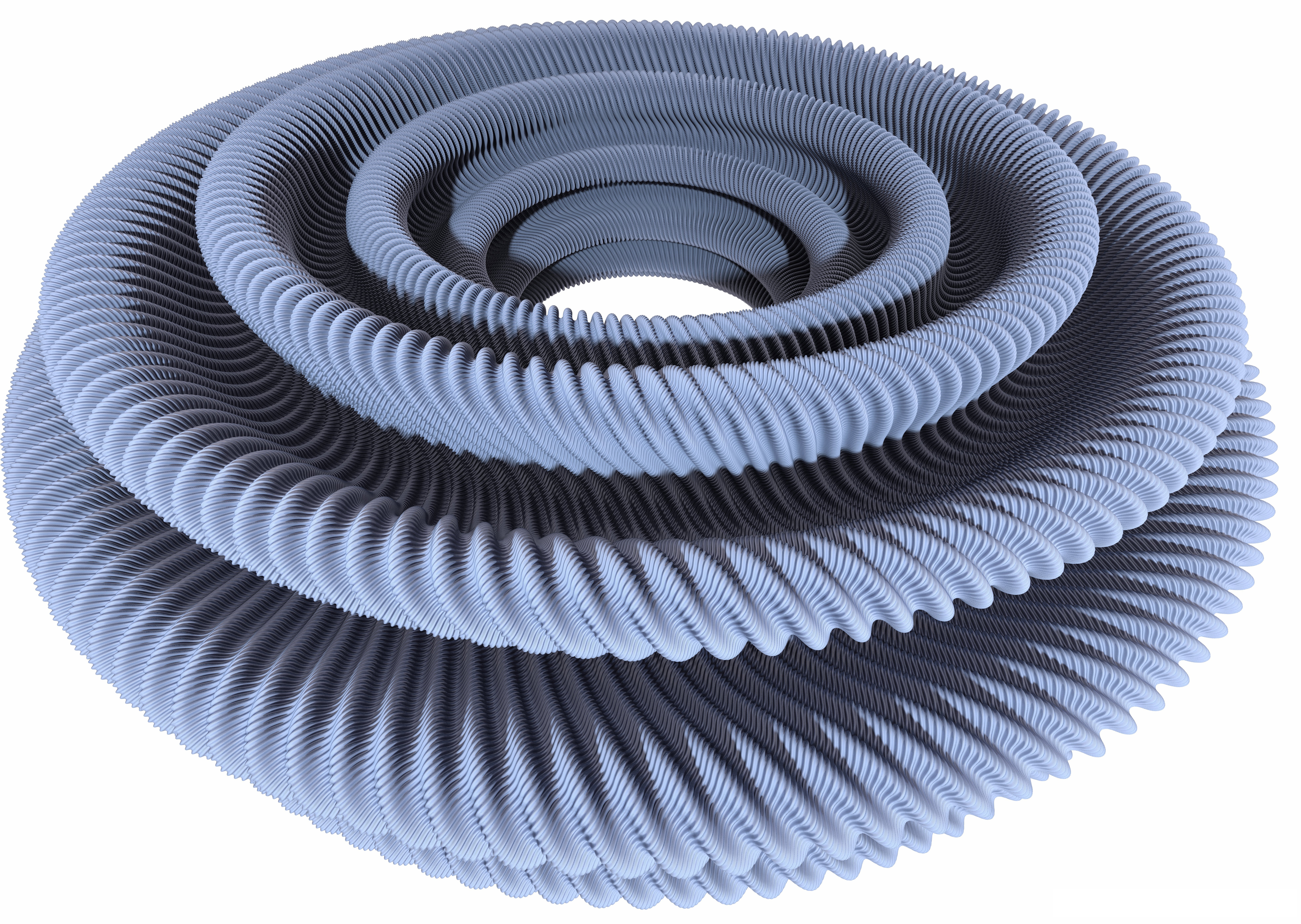}
\caption{A $C^1$ isometric embedding of a two-dimensional flat torus into three dimensional Euclidean space has a rough, corrugated looking surface as depicted in \cite{borrelli_pnas}.} 
\label{fig:cupcake}
\end{figure}

A flat torus is the quotient of $\R^n$ by a full-rank lattice with Riemannian metric induced by the standard Euclidean metric on $\R^n$.  It is a smooth and compact Riemannian manifold, whose Riemannian curvature tensor is identically zero. For the sake of completeness and inclusivity, we recall 

\begin{defn} \label{def:lattice} 
An $n$-dimensional (full rank) \em lattice \em $\G \subset \R^n$ is a set which can be expressed as $\G:=A\Z^n$ for an invertible $n \times n$ matrix $A$ with real coefficients. The matrix $A$ is called a \em basis matrix \em of $\G$.  The lattice defines the \em flat torus \em $\mathbb{T}_\G = \R^n / \G$, with Riemannian metric induced by the Euclidean metric on $\R^n$.  
\end{defn} 

The Laplace eigenvalue problem in this context is to find all functions defined on $\R^n$ for which there exists $\lambda \in \C$ such that   
\begin{align} \label{eq:eigPb}
    \Delta f(x) = \lambda f(x), \quad f(x+\ell) = f(x) \textrm{ and }\nonumber\\
    \nabla f(x+\ell) = \nabla f(x) \quad \forall \ell \in \G \textrm{ and } x \in \R^n.
\end{align}
Our sign convention for the Euclidean Laplace operator is 
\[ \Delta = - \sum_{i=1} ^n \frac{\pa^2}{\pa x_i ^2}.\] 
If $f$ is a solution to the Laplace eigenvalue problem, it is an \em eigenfunction. \em   With a bit of functional analysis \cite{cour-hil} one can prove that the corresponding eigenvalue is non-negative. The simplest case is a one-dimensional lattice, which consists of all integer multiples of a number, $\ell$.  The Laplace eigenvalue problem is then to find all functions which satisfy 
\[ - f''(x) = \lambda f(x), \quad f(x+k \ell) = f(x), \quad f'(x+k\ell) = f'(x), \quad \forall k \in \Z.\] 
Using calculus, the functions 
\[ f_n (x) := e^{2\pi i n x/\ell}, \quad n \in \Z\] 
are eigenfunctions, with corresponding eigenvalues 
\[ \lambda_n = \frac{4\pi^2 n^2}{\ell^2}.\] 
To prove that the functions $\{f_n\}_{n \in \Z}$ are \em all \em of the eigenfunctions up to scaling, it suffices to prove that they constitute an orthogonal basis for the Hilbert space $\cL^2$, as demonstrated in \cite{folland_fourier}.   The set of all of its Laplace eigenvalues with multiplicities is known as its \em spectrum.  \em  

In physics, the Laplace eigenvalue problem is a crucial step in solving the wave equation, as solutions can be expressed in terms of the Laplace eigenfunctions and eigenvalues.  The Laplace eigenvalues are in bijection with with the resonant frequencies of solutions to the wave equation.   For this reason, Kac \cite{kac} would say that \em spectral invariants, \em which are all quantities that are entirely determined by the spectrum, are \em audible. \em  We therefore ask \em can one hear the shape of a flat torus?  \em

\subsection{Three perspectives}\label{ss:three}   
We are interested in the question:  if two flat tori are isospectral, then are they isometric?  This has an equivalent formulation in both number theoretic terms as well as purely geometric terms as summarized in Table \ref{table:3eq} . This observation is crucial to obtain a thorough investigation.  
\begin{table}[H]
    \begin{tabular}{|l|l|} 
    \hline
    Analysis       & \makecell{\textit{Does the spectrum of the Laplace operator}\\ \textit{determine the geometry of flat tori?}}                            \\ \hline
    Number Theory & \makecell{\textit{Are positive definite quadratic forms determined}\\ \textit{by their values over the integers counting multiplicity?}}      \\ \hline
    Geometry        & \makecell{\textit{Do the lengths of points of a lattice with multiplicity}\\ \textit{determine the lattice itself up to congruency?}} \\ \hline
    \end{tabular}\caption{}
    \label{table:3eq} 
\end{table}
The first perspective gives the subject physical motivation. The second perspective is number theoretical, thereby opening the possibility to employ powerful  techniques from analytic number theory.  The third perspective is a more intuitive and purely geometric question about lattices. It is unfortunate that the mathematical language is quite different when the problem is investigated from these different viewpoints in the sense that disparate fields do not cross-reference each other.  Consequently, we will take this opportunity to connect some of the different terminologies. We aim to provide interested readers with a more thorough understanding of the question by studying it from all three perspectives.  


\subsection{Organization} \label{s:organization}   
In \S \ref{s:ingredients} we collect the essential ingredients required to investigate flat tori and their spectra.  We conclude that section by combining some of these ingredients to prove that isospectral rectangular flat tori are isometric, as are isospectral Euclidean boxes.  In \S \ref{s:kitchen} we introduce useful techniques, or key kitchenware, for investigating the spectrum and geometry of flat tori. We explore in \S \ref{s:differentdonuts} famous examples of non-isometric isospectral flat tori. Section \S \ref{s:schiemann} is dedicated to popularizing the fundamental yet not widely known theorem of Schiemann. We conclude in \S \ref{s:conjectures} with a collection of conjectures and open problems.  This includes a discussion of the open question:  how  \em many \em flat tori can be mutually isospectral and non-isometric in a given dimension?

\section{Essential ingredients and a case study} \label{s:ingredients} 
 We begin with the key ingredients in the mathematics of flat tori and their spectra.  
 
 \subsection{Lattices, congruency, and isometry} \label{ss:lattices_congruency_isometry} 
The general linear group is the set of all invertible $n \times n$ matrices with real coefficients, denoted $\GL_n (\R)$.  The unimodular group  $\GL_n(\Z) \subset \GL_n(\R)$ consists of those matrices $M \in \GL_n(\R)$ such that both $M$ and $M^{-1}$ have integer coefficients; necessarily all elements of $\GL_n(\Z)$ have determinant equal to $\pm 1$.  The orthogonal group $\oO_n(\R) \subset \GL_n(\R)$ consists of those matrices whose inverse matrix and transpose matrix are equal.  The following lemma summarizes basic facts about the orthogonal group.  Its proof is straightforward and therefore omitted.  

\begin{lemma}\label{le: orthl}  A matrix is an element of $\oO_n(\R)$  if and only if its column vectors form an orthonormal basis of $\R^n$. Further, $A\in \mathrm{O}_n(\RR)$ if and only if $Ax\cdot Ay=x\cdot y$ for any $x,y\in \RR^n$.   A matrix 
$C\in \oO_n(\R)$ if and only if the linear transformation $C:\R^n \to \R^n$ defined by $C(v) = Cv$ maps any orthonormal basis of $\R^n$ to another orthonormal basis of $\R^n$. 
\end{lemma}

\begin{defn}[Lattices of arbitrary rank, trivial lattices, sublattices, and integer lattices] 
A set $v_1,\ldots,v_k\in \R^n$ of linearly independent vectors define the \em k-rank lattice,  \em also called a $k$-dimensional lattice:  
\[ \Gamma := v_1\Z+\cdots +v_k\Z = \left\{ \sum_{j=1} ^k z_j v_j : z_j \in \Z \;\forall j \right\}. \]  
The matrix whose column vectors are equal to $v_1, \ldots, v_k$ is a \em basis matrix. \em Full-rank lattices are n-rank lattices in $\R^n$.  A \em trivial lattice \em is a set whose only element is $0$. A \textit{sublattice} $\Lambda$ of a lattice $\G$ is a lattice such that $\Lambda\subseteq \G$. An \em integer lattice \em is a sublattice of $\Z^n$.  
\end{defn}

We next collect basic facts about lattices.  
\begin{prop}\label{prop:sublattice}
Let $\Lambda$ and $\G$ be full-rank lattices and let $A_\Lambda,A_\G$ be corresponding bases.
\begin{enumerate}
    \item $\Lambda$ is a sublattice of $\G$ if and only if $A_\Lambda=A_\G V$ for some $V\in \mathbb{Z}^{n\times n}$.
    \item If $\Lambda\subseteq\G$, then $\det(A_\Lambda)/\det(A_\G)\in \mathbb{Z}\setminus\{0\}$ and 
    
   \[ \big(\det(A_\Lambda)/\det(A_\G)\big)\G\subseteq \Lambda. \] 
   \item Two matrices $A_1$ and $A_2$ in $\GL_n(\R)$ are both bases for the same lattice  if and only if there is a matrix $B \in \GL_n (\Z)$ such that $A_2 = A_1 B$.
    \item If $\Lambda\subseteq \G$, then the \textit{index} of $\Lambda$ in $\G$ as a subgroup, denoted $[\G:\Lambda]$, is equal to $|\det(A_{\Lambda})/\det(A_{\G})|$.
       \item A non-trivial additive subgroup $\G \subset \R^n$ is discrete if and only if it is a k-rank lattice for some $1\leq k \leq n$.  
\end{enumerate}
\end{prop}
\begin{proof} Statements (1), (2), and (3) are direct observations. A proof of (4) can be found in \cite[\S 1.2.4]{engel2004lattice}, and (5) is contained in \cite[p. 24]{neukirch2013algebraic}.  
\end{proof}

\begin{defn}[Congruent lattices and isometric flat tori] Let $\G_1 \subset \R^n$ and $\G_2 \subset \R^m$ be lattices. If $n=m$, then $\G_1$ and $\G_2$ are \em congruent \em if there is a $C \in \oO_n (\R)$ such that $\G_2 = C \G_1$.  If $n>m$, then $\G_1$ is congruent to $\G_2$ if $\G_2 \times \{0\}=C\G_1$ for an orthogonal matrix $C\in \oO_n(\R)$ and a trivial lattice in $\R^{n-m}$.  Two flat tori are \em isometric \em if they are isometric as Riemannian manifolds. 
\end{defn} 

We use the same notation $\cong$ for both isometry of flat tori as well as congruency of lattices as justified by the following theorem which shows that  flat tori are isometric if and only if the lattices that define them are congruent.    
\begin{theorem}[see p. 5 of \cite{berger1971spectre}]\label{th: RieIsom} Two flat tori are isometric in the Riemannian sense if and only if their associated lattices are congruent.  
\end{theorem} 

Necessary and sufficient conditions for full rank lattices to be congruent are given in the following lemma.  
\begin{lemma}\label{le:cong} Let $\G_1=A_1\Z^n$ and $\G_2=A_2\Z^n$ be two full-rank lattices.  Then 
\[ \G_1\cong\G_2 \iff CA_1=A_2B\] 
for some $B\in \GL_n(\ZZ),C\in \oO_n(\R)$.
\end{lemma}

\subsection{The spectrum of a flat torus} \label{ss:donut_song} 
The spectrum of a flat, $\R^n/\Gamma$, is in bijection with the lengths of the vectors in the dual lattice of $\Gamma$.

\begin{defn}[Dual lattice] \label{def:dual_lattice} For a full-rank lattice $\G \subset \R^n$, its \em dual lattice \em is defined to be 
\[ \G^*:=\{\ell\in \R^n:\ell\cdot \g \in \Z, \:\:\forall \g \in \G\}. \]  
\end{defn} 

It is straightforward to show that the dual lattice of a full-rank lattice $\G$ is itself a lattice, and there is a natural bijection between $\G^*$ and $\mathrm{Hom}(\G,\ZZ)$, justifying the name dual lattice.  If a full-rank lattice $\G$ has basis matrix $A$, then $(A^{-1})^T$ is a basis matrix for the dual lattice.

\begin{theorem}[The Laplace spectrum of a flat torus]\label{th: LaplaceSpec} The eigenvalues of a flat torus $\R^n / \G$ are precisely $4\pi^2 ||\ell||^2$ such that $\ell$ is an element of the dual lattice, $\G^*$.  The multiplicity of such an eigenvalue is the number of distinct elements of $\G^*$ that have the same length as $\ell$.  The eigenspace is spanned by the functions 
\[ \{u_\ell (x) = e^{2\pi i x \cdot \ell} \}_{\ell\in \G^*}. \]  
The collection of eigenvalues, counted with multiplicity, is \em the spectrum \em of the flat torus.
 \end{theorem} 

\begin{proof} For any $\ell$ in the dual lattice, the function $u_\ell (x) = e^{2\pi i x \cdot \ell}$ satisfies both the Laplace eigenvalue equation on $\R^n$ as well as $u_\ell(x+\g) = u_\ell(x)$, and $\nabla u_\ell (x+\g) = \nabla u_\ell (x)$ for all $\g \in \G$.  Moreover, one can demonstrate that the functions $u_\ell$ are an orthogonal basis for the Hilbert space $\cL^2$ on the torus, for example by showing they are the solutions of a regular Sturm-Liouville problem \cite{folland_fourier}.    
\end{proof}

\begin{defn}[Isospectrality] Two flat tori $\R^n/\G$ and $\R^n/\Lambda$ are \em isospectral \em if they have the same Laplace spectrum.
\end{defn} 

One can readily check that if two flat tori are isometric, then they are isospectral.  Is the converse true?  We explore this question in \S \ref{s:schiemann}.

\subsubsection{Poisson's summation formula} \label{s:poissonsmagic} 
Poisson's summation formula is a powerful tool because it equates a purely analytical object with a purely geometric one.  Two important geometric ingredients in the formula are the volume of the flat torus and its length spectrum.

\begin{defn}[Volume]
For a flat torus $\mathbb{T}_\G = \R^n / \G$, with full-rank lattice $\G=A\Z^n$, the \em volume \em of $\mathbb T_\G$ (with respect to the flat Riemannian metric induced by the Euclidean metric on $\R^n$) is equal to 
\[ \vol(\mathbb{T}_\G):=|\det(A)|.\] 
\end{defn}

It is straightforward to show that the volume is independent of the choice of basis matrix.  If $\G$ is an integer lattice, then  $\vol(\G) \Z^n$ is a sublattice of $\G$.

\begin{defn}[The length spectrum] For a flat torus $\mathbb{T}_\G = \R^n / \G$, with lattice $\G=A\Z^n$, the \em length spectrum \em  of $\mathbb T_\G$ (with respect to the flat Riemannian metric induced by the Euclidean metric on $\R^n$) is equal to the collection of lengths of closed geodesics, counted with multiplicity, and denoted by $\cL_\G$. This length spectrum is also equal to the collection of lengths of lattice vectors, $||\gamma||$ for $\gamma \in \Gamma$, counted with multiplicity, which is how we define the length spectrum of the lattice $\Gamma$.\footnote{We note that there is a related question of \textit{length-equivalence} of lattices \cite{oishi2020positive}, in which one considers the set of lengths of lattice vectors, ignoring their multiplicity.} 
\end{defn}

The Poisson summation formula \cite[p. 125]{cohn2014formal} equates a sum over the Laplace spectrum with a sum over the length spectrum, thereby relating these two spectra. 

\begin{theorem}[Poisson summation formula]\label{th: poiss} For an full-rank lattice $\G \subset \R^n$ the following series converge for any $t\in (0,\infty)$ and satisfy
\begin{align} \label{eq:poisson1} 
    \sum_{\g^*\in \G^*}e^{-4\pi^2 \|\g^*\|^2t} & =\frac{\vol(\G)}{(4\pi t)^{n/2}}\sum_{\g\in \G}e^{-\|\g\|^2/4t}\\
    \sum_{\g\in \G}e^{-\|\g\|^2/4t} & = \frac{(4\pi t)^{n/2}}{\vol(\G)}\sum_{\g^*\in \G^*}e^{-4\pi^2 \|\g^*\|^2t}  \label{eq:poisson2} 
\end{align}
\end{theorem}

The first series in Poisson's summation formula is a spectral invariant known as the \em heat trace.\em

\begin{defn}[Heat trace] \label{def:theta} The \em heat trace \em of a flat torus $\T_\G$ (and lattice $\G$) is defined for $t>0$ by
\[ \sum_{\g^* \in \G^*} e^{-4\pi^2 ||\g^*||^2 t} = \sum_{k \geq 0} e^{-\lambda_k t},\] 
where the eigenvalues $\lambda_k$ are ordered as $0 = \lambda_0 < \lambda_1 \leq \lambda_2 \uparrow \infty$ and counting multiplicities.  
\end{defn} 

A closely related spectral invariant is the \em theta series.  \em 

\begin{defn}[Theta series] \label{def:theta}  Let $\G$ be a lattice. Then we define the \em theta series \em of the lattice (and flat torus) as  
\[ \theta_\Gamma(z):=\sum_{\g \in \G} e^{i \pi z ||\g||^2}, \quad z \in \C \textrm{ with } \Ima z > 0. \] 
\end{defn}

By Poisson's summation formula, the lattices $\G,\Lambda$ are isospectral if and only if their dual lattices $\G^*,\Lambda^*$ are isospectral.
The Poisson summation formula shows that the dimension, volume, and length spectrum are all spectral invariants.  Moreover, it establishes the equivalence of the analytical and geometrical formulations in \S \ref{ss:three}.  

\begin{cor} \label{cor:dim_vol} 
The following are equivalent:
\begin{enumerate} 
\item The flat tori $\mathbb{T}_\G$ and $\mathbb{T}_{\Lam}$ are isospectral.
\item The flat tori $\mathbb{T}_\G$ and $\mathbb{T}_{\Lam}$ have identical heat traces.
\item The flat tori $\mathbb{T}_\G$ and $\mathbb{T}_{\Lam}$ have identical theta series. 
\item The flat tori $\mathbb{T}_\G$ and $\mathbb{T}_{\Lam}$ have identical length spectra. 
\end{enumerate} 
If any of the above hold, then the flat tori have identical dimension and volume.
\end{cor} 
\begin{proof} The implication (1) $\implies$ (2) follows immediately from the definition.  To prove the converse, proceed inductively to prove that the heat trace determines all of the eigenvalues and their multiplicities by analyzing the asymptotic behavior of the heat trace as $t \nearrow \infty$.  The equivalence of (1) and (4) is an immediate consequence of Poisson's summation formula.  The heat trace is obtained by evaluating the theta series of the dual lattice at $z=4\pi i t$ for $t>0$.  Consequently, identical theta series imply that the dual lattices have identical heat traces and are therefore isospectral.  Isospectrality of lattices is equivalent to isospectrality of their dual lattices, showing that (3) implies (1).  The converse follows from Poisson's summation formula.  We have therefore established that (1) is equivalent to (2), (4), and (3).  To conclude that isospectral tori have identical dimension and volume, investigate the asymptotic behavior for $t$ approaching $0$ and for $t$ approaching infinity in Poisson's formula.
\end{proof}


\subsection{Quadratic forms equate the number theoretic formulation} \label{ss:quadratic_forms} 
We have seen that the analytical and geometrical formulations in \S \ref{ss:three} are equivalent with help of Poisson's powerful summation formula.  To equate these formulations in number theoretic language, we collect several facts about quadratic forms.  In \S \ref{ss:three_revisited}, we will use these facts to associate a quadratic form to a flat torus and identify its representation numbers with the spectrum. Quadratic forms will also be essential in \S \ref{s:schiemann}.

\begin{defn} \label{def:pqf} A \em quadratic form, \em  $q$, of $n$ variables is a homogeneous polynomial of degree $2$.  If $q(x) \geq 0$ for all nonzero $x \in \R^n $, it is \em positive semi-definite, \em and if the inequality is strict, then the form is \em positive definite.  \em We may equivalently refer to quadratic forms of $n$ variables as $n$-dimensional quadratic forms.
\end{defn} 

There is a well-known natural bijection between positive definite quadratic forms and symmetric matrices.  For any quadratic form $q$ of $n$ variables, there is a unique symmetric $n \times n$ matrix $Q$, known as the \em associated matrix, \em such that 
\[ q(x) = x^T Q x, \quad \forall x \in \R^n. \] 
The matrix $Q$ is positive (semi-)definite if and only if the quadratic form $q$ is. Although it is a slight abuse of notation, we hope the reader will pardon our identification of quadratic forms with their associated matrices in the following  

\begin{defn} \label{def:sets_of_forms} The notation $\mathcal{S}_{>0}^n$, respectively $\mathcal{S}_{\ge 0}^n$, is the set of $n \times n$ positive definite, respectively semi-definite, matrices and is also identified with the set of positive definite, respectively semi-definite, quadratic forms of $n$ variables.  A quadratic form is \em rational \em if the entries of its associated matrix are all rational.  A quadratic form is \em even \em if the entries of its associated matrix are integers, and the diagonal consists of even numbers.  If a quadratic form $q$ is positive definite, then its \em dual form, $q^*$, \em is defined by 
\[ q^*(x) := x^T Q^{-1} x, \quad \forall x \in \R^n. \]
\end{defn}

The matrix $Q$ associated with a positive definite quadratic form admits a \em Cholesky factorization.  \em 
\begin{theorem}[Theorem 11.2 in \cite{serre2000matrices}] 
Assume that $Q$ is a symmetric $n \times n$ real matrix such that $x^TQx>0$ for all non-zero $x\in \RR^n$.  Then there is an invertible matrix $A$ with $Q=A^TA$.  This is known as a \em Cholesky factorization. \em The matrix $A$ is unique up to left-multiplication with an orthogonal matrix. \end{theorem} 

\begin{proof} 
A symmetric matrix $Q$ can be diagonalized as $D=U^T Q U$, with $U^T = U^{-1}$ an orthogonal matrix.  Since $Q$ is positive definite, each diagonal entry of $D$ is positive, and hence there is a well-defined diagonal matrix $S$ with positive diagonal entries such that $D=S^2 = S S^T$.  Define $A=(USU)^T$.  For the uniqueness, observe that if $B^T B = Q$, then $(A^T)^{-1} B^T =(B A^{-1})^T = A B^{-1} = (B A^{-1})^{-1}$, so $B=CA$ for the orthogonal matrix $C=BA^{-1}$. 
\end{proof}

Motivated by the Cholesky factorization, we define \em underlying lattices \em of a quadratic form.  

\begin{defn}[Underlying lattice] \label{def:underlying}  
For a positive definite $n \times n$ matrix $Q$ with Cholesky factorization $Q=A^TA$, we say that $A\ZZ^n$ is an \textit{underlying lattice} of both $Q$ and the associated quadratic form. 
\end{defn} 

Cholesky factorization is unique up to left-multiplication with an orthogonal matrix, and all underlying lattices of a given positive definite quadratic form are congruent.

\begin{defn}[Integral equivalence] \label{defn:int_equiv_qf} Two quadratic forms $q$ and $p$ on $\R^n$ are \em integrally equivalent \em if their associated matrices $Q$ and $P$ satisfy $B^T Q B = P$, for a unimodular matrix $B$.   
\end{defn}

The following proposition collects several useful facts about quadratic forms. It can be proven using the spectral theorem for symmetric matrices. 
\begin{prop}\label{PosQ} Let $Q$ be a real symmetric $n\times n$ matrix. Let $\lambda_{\min}$ be its smallest eigenvalue and $\lambda_{\max}$ its biggest.  

\begin{enumerate}
    \item $x^TQx>0$ for all $x\in \Z^n\setminus\{0\}$ if and only if $x^TQx>0$ for all $x\in \Q^n\setminus\{0\}$.
    
    \item $x^TQx>0$ for all $x\in \Z^n\setminus\{0\}$ implies $Q\in\mathcal{S}_{\ge 0}^n$.
    
    \item $Q\in \mathcal{S}_{\ge 0}^n$ if and only if $Q$ has only non-negative real eigenvalues.
    
    \item $Q\in \mathcal{S}_{>0}^n$ if and only if $Q$ has only positive real eigenvalues.
    
    \item $\lambda_{\min}\|x\|^2\le x^TQx\le \lambda_{\max}\|x\|^2 $ for any $x\in \R^n$.
    
    \item $Q\in \mathcal{S}_{\ge0}^n$ implies $Q=E^TE$ for some $n \times n$ real matrix $E$. 
    
    \item $Q\in \mathcal{S}_{>0}^n$ if and only if $Q\in \mathcal{S}_{\ge 0}$ and $Q$ is of full rank.
    
    \item Viewed as a quadratic form, the image of $Q\in \mathcal{S}_{>0}^n$ over $\Z^n$ is discrete, and all multiplicities are finite.
\end{enumerate}
\end{prop}

\subsubsection{Representation numbers of quadratic forms} \label{ss:three_revisited} 
The connection between the spectra of flat tori and quadratic forms is obtained using the \em representation numbers \em of quadratic forms.  The representation numbers are the image of $\Z^n$ under the quadratic form, taking into account multiplicities.  

\begin{defn}[Representation Numbers] If $q$ is an $n$-dimensional positive definite quadratic form, its \em representation numbers \em are defined as follows for $t\in \R_{\ge 0}$ 
\[ \mathcal{R}(q,t):=\#\{x\in\Z^n:q(x)=t\}.\] 
We may also consider a subset $X \subset \Z^n$ and define 
\[ \mathcal{R}_X(q,t):=\#\{x\in X:q(x)=t\}.\] 
\end{defn}

A convenient reduction for verifying whether two quadratic forms have the same representation numbers is obtained by defining 
 \begin{equation} \label{eq:zn_star} \begin{gathered} \Z_*^n:=  \left\{x\in \Z^n\setminus\{0\}:  \gcd(x_1,\ldots,x_n)=1, \right . \\ 
 \left . \textnormal{ and the last non-zero coordinate is positive}\right\}.
\end{gathered}  \end{equation} 
Then, one can show that for two quadratic forms $q_1$ and $q_2$, for any $t \geq 0$, 
\[ \mathcal{R}(q_1,t) = \mathcal{R}(q_2,t) \iff   \mathcal{R}_X(q_1,t) = \mathcal{R}_X(q_2,t) \textrm{ for } X = \Z_* ^n.\] 

A positive definite quadratic form $q$ has a collection of underlying lattices, all of which are congruent.  Consequently, the associated flat tori are all isometric and therefore also isospectral. On the one hand, for a full-rank lattice $A \Z^n$, for any unimodular matrix $G \in \GL_n(\Z)$, $A G \Z^n$ and $A \Z^n$ are the same lattice. The quadratic form with matrix $(AG)^T (AG)$ is not necessarily the same as the quadratic form with matrix $A^T A$.  These two quadratic forms are, however, integrally equivalent, and any two $n$-dimensional positive definite quadratic forms that are integrally equivalent have identical representation numbers for all $t \geq 0$.  

For a flat torus $\R^n/\G$ with $\G=A \Z^n$ we associate the equivalence class of quadratic forms that are integrally equivalent to $A^T A$.  There is a natural bijection between the length spectrum and the representation numbers of this equivalence class, taking for some $x\in \ZZ^n$, $\|Ax\|$ to $q(x)=x^TA^TAx=\|Ax\|^2$.  It then follows from Corollary \ref{cor:dim_vol} that two flat tori are isospectral if and only if their representation numbers associated in this way are identical.  They are isometric if and only if their equivalence classes of quadratic forms are in fact identical. We therefore define isospectrality for flat tori, lattices, and quadratic forms.

\begin{defn}[Isospectrality of lattices and quadratic forms]  Two lattices $\Gamma_i \subset \R^n$, $i=1$, $2$, are isospectral if and only if the flat tori $\R^n/\Gamma_i$ have identical Laplace spectra, or equivalently, they have identical length spectra.  Two quadratic forms are isospectral if and only if they have identical representation numbers. 
\end{defn} 

The precise number theoretic formulation in \S \ref{ss:three} is then:  is a quadratic form uniquely determined by its representation numbers, up to integral equivalence? We will see that the answer to the question depends on the dimension.


\subsection{Constructing and deconstructing lattices with implications for isospectrality} \label{ss:biggerdonuts}  
One can construct higher dimensional flat tori by taking products of lower dimensional ones, or equivalently, one can build a full-rank lattice by summing lower rank lattices. This technique has been historically important, which will be apparent in Section \ref{s:differentdonuts}.  We will also use it to give an elegant proof of the lower bound for the number of isospectral but non-isometric flat tori in each dimension in \S \ref{ss: asymp}.  If a flat torus has been built as a product of lower dimensional flat tori it is \em reducible, \em and if not, it is \em irreducible.\em\footnote{Irreducible may also be termed \em indecomposable.\em} We use the same terminology for the lattice that defines the flat torus. Reducibility of lattices have connections to root systems and Dynkin diagrams that are studied in the theory of Lie groups; see for example \cite[p. 217]{hall2015lie}.  A key ingredient in the definition of reducibility is the Minkowski sum.  

\begin{defn}[Minkowski sum] 
Let $A$ and $B$ be two non-empty sets in $\R^n$. Their \em Minkowski sum, \em denoted $A+B$ is 
\[ A+B:= \{ a+b \in \R^n : a \in A, b \in B \}.\] 
The product 
\[ A \cdot B := \{ a \cdot b : a \in A, \quad b \in B\} ,\] 
where $a \cdot b$ is the scalar product in $\R^n$.   
If $A\cdot B=\{0\}$, $A+B = A\oplus B$ is a \textit{direct sum}. 
\end{defn}

\begin{defn}[Reducible \& irreducible lattices and flat tori]  \label{defn:irreducibility} Let $\G\subseteq \RR^n$ be a non-trivial lattice. Then $\G$ is reducible if $\G = \G_1 \oplus \G_2$ for two non-trivial lattices $\G_i$, $i=1,2$. Otherwise $\G$ is irreducible. The associated flat torus  $\R^n/\G$ is reducible or irreducible if $\G$ is reducible or irreducible, respectively.  
\end{defn}

By the definition of irreducibility, a full-rank lattice $\G \subset \R^n$ can be decomposed as a sum of irreducible sublattices $\G_i$, so that 
\[ \G = \G_1 \oplus \ldots \oplus \G_k,\] 
for some $k \geq 1$.  This is known as the \em irreducible decomposition, \em
and is unique up to re-ordering. Kneser investigated a more general setup in \cite{kneser1954theorie}.  The following lemma shows that an equivalent way to define reducibility and irreducibility is through products.

\begin{lemma} \label{le:irreducibility} 
A lattice $\G \neq \{0\}$  is reducible if and only if it is congruent to a lattice of the form $\G_1\times \G_2$ where $\G_1,\G_2$ are of dimensions at least $1$. Conversely, if there are no such $\G_1$ and $\G_2$, then $\G$ is irreducible.  
\end{lemma}

\begin{proof}  Without loss of generality, we may assume that $\G$ is a full-rank lattice in $\R^n$.  Then, $\G$ is reducible if and only if there are two non-trivial sublattices $G_1$ and $G_2$ such that $\G = G_1 \oplus G_2$.  These induce an orthogonal decomposition of $\R^n$ into two subspaces of dimensions $k_1$ and $k_2$, the ranks of $G_1$ and $G_2$, respectively, with $k_1+k_2=n$.  There is an isometry, denoted $\Phi_i$, between each of these subspaces and $\R^{k_i}$, respectively, for $i=1,2$.  Consequently, $G_i \cong \Phi_i(G_i) =: \Gamma_i \subset \R^{k_i}$, for each $i=1, 2$, and $\G = G_1 \oplus G_2 \cong  \Phi_1(G_1) \times \Phi_2(G_2) = \Gamma_1 \times \Gamma_2$.   
\end{proof} 

We note that a lattice is irreducible if and only if its dual lattice is irreducible.  The following lemma shows that re-arranging the constituents in a product of lattices results in a congruent lattice.

\begin{lemma} \label{le:permutations} The product of lattices $\G_1\times \cdots\times \G_m$ is congruent to the product $\G_{\sigma(1)}\times \cdots\times \G_{\sigma(m)}$ for any permutation $\sigma\in S_m$.
\end{lemma}

\begin{proof} 
We outline the key ideas.  Let $A=[a_j]$ be an $n \times n$ matrix.  Right multiplication by elements of $\GL_n(\Z)$ can re-order the columns of $A$ in any desired way.  Left multiplication by elements of $O_n(\R)$ can re-order of the rows of $A$ in any desired way.  
\end{proof}

\begin{prop} \label{le:congruenceproducts} Assume that a lattice $\G$ in $\R^n$ can be decomposed into an orthogonal sum of sublattices 
\[\G= \G_1\oplus\cdots\oplus \G_k.\]
Then $\Lambda \cong \G$  if and only if $\Lambda$ is a direct sum of $k$ sublattices $\Lambda_i \cong \G_i$.  
\end{prop}

\begin{proof} Without loss of generality we assume that $\Lam \subset \R^n$, and that $\G$ is full-rank.  If $\Lam \cong \G$, there is $C \in \oO_n(\R)$ such that $\Lam = C\G = C(\G_1\oplus\cdots\oplus \G_k )$.  Since $C$ preserves orthogonality, $\Lam = C\G_1 \oplus \cdots \oplus C \G_k$.  Thus, defining $\Lam_i := C \G_i \cong \G_i$ completes the proof in this direction. For the other direction, assume that $\Lambda_i \cong \G_i$ for each $i=1, \ldots, k$, with $\Lam := \Lam_1 \oplus \cdots \oplus \Lam_k$.  Then, there exist orthogonal transformations $C_i \in \oO_n(\R)$ such that $C_i \G_i = \Lam_i$.  By orthogonality, $\R^n$ admits an orthogonal decomposition into $k$ subspaces, the $i^{th}$ subspace containing $\G_i$.  Let $\Pi_i$ be orthogonal projection onto the $i^{th}$ subspace.  We therefore define \phantom\qedhere
\[\pushQED{\qed}  C := \sum_{i=1} ^k C_i \circ \Pi_i \in \oO_n (\R), \quad C(\G) = \Lam_1 \oplus \cdots \oplus \Lam_k = \Lam \implies \G \cong\Lam. \qedhere
\popQED\]
\end{proof}

An immediate consequence is 

\begin{cor} \label{cor:irreducible} 
Two products of irreducible lattices are congruent: 
\[ \G_1\times \cdots \times \G_k \cong \Lam_1\times \cdots \times \Lam_{k'} \]  if and only if $k=k'$, and up to reordering $\G_i \cong \Lam_i$. 
\end{cor}

\begin{prop}\label{prop:cong1} Two lattices $\G,\Lambda$ are congruent if and only if $\G^{n},\Lambda^{n}$ are congruent for some $n\in \Z_{\geq 2}$. \end{prop}

\begin{proof} If $\G$ and $\Lambda$ are congruent, we leave it to the reader to show that $\G^n$ and $\Lambda^n$ are congruent. On the other hand, we may up to congruence consider irreducible decompositions
\[ \G_1^{n}\times \cdots \times \G_k^{n} \quad\&\quad \Lambda_1^{n}\times \cdots \times \Lambda_{k'}^{n}.  \] 
By Corollary \ref{cor:irreducible}, $nk=nk'$, and therefore $k=k'$.  By possibly re-ordering and re-naming, without loss of generality, $\Lambda_i \cong \G_i$ for each $i=1, \ldots, k$. 
\end{proof}

The following theorem is an immediate consequence of the definitions and preceding results together with Witt's cancellation theorem.  Witt's theorem \cite{witt_cancel} shows that if the same term appears in a product, one can \em cancel \em that term. 

\begin{theorem} \label{th:inheritance1} Fix two congruent lattices $\Lam$ and $\Lam'$.  Two lattices $\G$ and $\G'$ are congruent if and only if $\G \times \Lambda$ is congruent to $\G'\times \Lam'$.
\end{theorem}

With help from the theta series we obtain a similar isospectral cancellation result for isospectrality.  

\begin{lemma}  \label{le:thetaproducts} 
If $\G_1$ and $\G_2$ are lattices and $\G = \G_1 \times \G_2$, then their theta series satisfy $\theta_\G = \theta_{\G_1}\theta_{\G_2}$. As a consequence, for an arbitrary lattice $\Lambda$, $\Lambda\times \G_1$ is isospectral to $\Lambda \times \G_2$ if and only if $\G_1$ and $\G_2$ are isospectral.
\end{lemma}

\begin{proof} The first statement is left to the reader; it can be demonstrated using the Pythagorean theorem. Consider $n$-dimensional lattices $\Gamma_1,\Gamma_2$ and an $m$-dimensional lattice $\Gamma$. By Corollary \ref{cor:dim_vol}  $\Gamma_1 \times \Gamma$ and $\Gamma_2 \times \Gamma$ are isospectral if and only if $\theta_{\Gamma_1\times\Gamma}=\theta_{\Gamma_2\times\Gamma}$. This is equivalent to $\theta_{\Gamma_1}\theta_{\Gamma}=\theta_{\Gamma_2}\theta_{\Gamma}$.  By Definition \ref{def:theta}, evaluating the theta series at $i \R_{>0}$, they are positive, so we may divide obtaining $\theta_{\Gamma_1}(iy)=\theta_{\Gamma_2}(iy)$ holds for all $y>0$.  By the identity theorem, since theta series are holomorphic in the upper half plane, $\theta_{\Gamma_1}$ and $\theta_{\Gamma_2}$ are identical.  By Corollary \ref{cor:dim_vol}, this is true if and only if $\Gamma_1$ and $\Gamma_2$ are isospectral. 
\end{proof}

The next lemma can be proven by combining the ingredients we have collected thus far to show that once we have isospectral non-isometric flat tori in dimension $n$,  then we have them in all higher dimensions.  

\begin{lemma}\label{le:schiemlemma} If there exist $k$ mutually isospectral and pairwise non-isometric flat tori in dimension $n$, then there exist $k$ mutually isospectral and pairwise non-isometric flat tori in all higher dimensions. 
\end{lemma}

\subsection{One can hear the shapes of rectangular flat tori and Euclidean boxes}\label{s:rectangular}   
A \em rectangular flat torus \em can be completely reduced to a product of intervals.     
\begin{defn}[Rectangular Lattices]  \label{def:rectangular} A \em rectangular lattice \em  $\G$ is a lattice that has a diagonal basis matrix.  The associated flat torus $\R^n/\G$ is a \em rectangular flat torus.  \em 
\end{defn}

If an n-rank lattice $\Gamma$ has a diagonal basis matrix, then there are scalars $\{ c_j \}_{j=1} ^n$ such that a basis for the lattice is $ \{ c_j e_j \}_{j=1} ^n$, 
where $e_j$ are the standard orthonormal basis vectors of $\R^n$.  Consequently, the rectangular lattice $\G \cong \Gamma_1 \times \ldots \times \Gamma_n$,   for the one-dimensional lattices $\Gamma_j = \Z c_j \cong \Z c_j e_j$.  

\begin{theorem} If two rectangular flat tori are isospectral then they are isometric. 
\end{theorem}

\begin{proof} 
The proof is by induction on the dimension.  If two rectangular flat tori are isospectral, then they are the same dimension, so they are both defined by full-rank lattices in $\R^n$ that have diagonal basis matrices.  The case $n=1$ follows immediately from the equality of the first positive eigenvalue that completely determines a one-dimensional flat torus.  So we now assume the theorem has been proven for dimensions up to some $n \geq 1$.  Assume that two $n+1$ dimensional rectangular flat tori are isospectral.  The rectangular tori are each defined by the products of the one-dimensional lattices 
\[\Gamma= \Gamma_1 \times \ldots \times \Gamma_{n+1}, \textrm{ and } \Lam = \Lam_1 \times \ldots \times \Lam_{n+1},\]
with each $\Gamma_j = \Z c_j \cong \Z c_j e_j$, for the standard unit vector $e_j$ and for some non-zero $c_j$.  The length of the shortest non-zero vector in $\Gamma$ is therefore the minimal $|c_j|$.  The length of the shortest non-zero vectors of $\Gamma$ and $\Lambda$ are the same by Corollary \ref{cor:dim_vol}.  By Lemma \ref{le:permutations} we can without loss of generality re-arrange these products to assume that $\Gamma_1 = \Z c_1 e_1$, and $\Lam_1 = \Z c_1 e_1$.  Consequently, these are congruent and isospectral.  By Lemma \ref{le:thetaproducts}, we therefore have that $\Gamma_2 \times \ldots \times \Gamma_{n+1}$ and $\Lam_2 \times \ldots \times \Lam_{n+1}$ are isospectral, and they are $n$-dimensional.  Consequently, by the induction assumption, they are congruent.  We therefore have by Theorem \ref{th:inheritance1} that $\Gamma$ and $\Lambda$ are also congruent, and consequently the flat tori they define are isometric.  
\end{proof}

Is there a rectangular flat torus that is isospectral to a non-rectangular flat torus? The answer will be revealed in Section \ref{ss: rect6}.

\subsubsection{One can hear the shape of a Euclidean box} \label{ss:boxes} 
Although we would expect the following result to be known, we are unaware of a reference in the literature and therefore include it here.  
A \em Euclidean box \em is a bounded domain in $\R^n$ that is the Cartesian product of $n$ bounded intervals.    

\begin{theorem} Assume that two Euclidean boxes are isospectral with respect to the Dirichlet boundary condition or the Neumann boundary condition.  Then the two boxes are isometric.  
\end{theorem}

\begin{proof} 
If two boxes are isospectral, then the same boundary condition must be taken on both boxes.  The Dirichlet boundary condition requires eigenfunctions to vanish on the boundary, whereas the Neumann boundary condition requires eigenfunctions to have vanishing normal derivative on the smooth components of the boundary.  In the Neumann case, $0$ is an eigenvalue, whereas in the Dirichlet case, the spectrum is strictly positive. Although the proof is similar to the one for rectangular flat tori, the flavor is a bit different.  For flat tori, one has the Poisson summation formula, which we do not have for Euclidean boxes.  We replace this ingredient with the explicit calculation of the eigenvalues.  We will prove the theorem for the Dirichlet boundary condition because the proof for the Neumann condition is completely analogous.  

In one dimension, the spectrum consists of $n^2 \pi^2/\ell^2$ for $n \in \N$, with $\ell$ the length of the one-dimensional box (interval). Thus, if one dimensional boxes are isospectral then they are isometric. Now assume the statement is true for all dimensions from $1$ to $n$ for some $n \geq 1$.  By Weyl's law \cite{weyl}, if the eigenvalues are listed in non-decreasing order, then the $k^{th}$ eigenvalue grows with $k \to \infty$ on the order of $k^{2/n}$ for an $n$-dimensional box.  As a consequence, if two boxes are isospectral, then they are the same dimension. Consider now an $n+1$ dimensional box.  The first two smallest eigenvalues are 
\[ \lambda_1 = \sum_{k=1} ^{n+1} \frac{\pi^2}{\ell_k^2}, \quad \lambda_2 = \frac{4\pi^2}{\ell_{max}^2} + \sum_{\ell_k \neq \ell_{max}} \frac{\pi^2}{\ell_k ^2}. \] 
Here, $\ell_{max}$ is the length of the longest side of the box.  If two boxes are isospectral, they have the same first two eigenvalues, as well as the same difference between these first eigenvalues, that is $3\pi^2/\ell_{max}^2$.  Consequently, $\ell_{max}$ is the same for both boxes.  The two boxes are therefore each respectively isometric to 
\[ B \times [0, \ell_{max}], \quad B' \times [0, \ell_{max}].\]   
Here, $B$ and $B'$ are boxes of dimension $n$.  For simplicity, set $\ell:=\ell_{max}$. Since the eigenvalues of $B \times [0, \ell]$ are equal to the sum of the eigenvalues of $B$ and the eigenvalues of $[0, \ell]$, the heat trace satisfies  $H_{B\times [0, \ell]}(t) = H_B(t) H_{[0, \ell]}(t)$ and similarly $H_{B'\times [0, \ell]}(t) = H_{B'}(t) H_{[0, \ell]}(t)$. The two heat traces are equal by isospectrality, and therefore $H_B=H_{B'}$, from which it follows that $B$ and $B'$ are isospectral. Since $B$ and $B'$ are dimension $n$, by induction they are also isometric.  By Theorem \ref{th:inheritance1}, the two boxes are respectively isometric to $B \times [0, \ell] \cong B' \times [0, \ell]$.
\end{proof}


\section{Key kitchenware for computation and construction}\label{s:kitchen} 
In the preceding section we introduced the essential ingredients to understand the geometry and spectra of flat tori.  Here we present useful tools for manipulating and combining these ingredients.  In this section, all lattices are full-rank.

\subsection{Congruency tests} \label{s:tech_cong}
Given two lattices $\G_1=A_1\ZZ^n$ and $\G_2=A_2\ZZ^n$ with two explicit bases, we can easily check if they are the same lattice, since this is the case if and only if $A_1^{-1} A_2$ is a unimodular matrix.  Checking whether or not the two lattices are congruent is not as simple.  The lattices are congruent if and only if there is an orthogonal matrix $C$ and a unimodular matrix $B$ such that $CA_1 B=A_2$.  There are infinitely many orthogonal matrices and infinitely many unimodular matrices, so checking congruency is a seemingly infinite task.   One way to conclude that lattices are not congruent is furnished by Proposition \ref{le:congruenceproducts}, from which we immediately obtain 

\begin{cor}\label{cor: red} If $\G$ is reducible and $\Lambda$ is not, then $\G$ is not congruent to $\Lambda$.
\end{cor}

A second method uses Corollary \ref{cor:dim_vol} together with the fact that orthogonal transformations preserve the scalar product.  

\begin{cor} Let $\G_1,\G_2$ be two lattices.  For $s>0$, let $S_i(s)$ be the sets of all vectors in $\G_i$ of length $s$. If there exists an $s>0$ such that the number of elements in $S_i(s)$ are not equal, then $\G_i$ are not congruent.  Let $P_i(s) = \{ a \cdot b : a, b \in S_i(s) \}$.  If there exists an $s>0$ such that $P_i(s)$ are not equal, then $\G_i$ are not congruent.
\end{cor}

A third method uses the equivalence classes of quadratic forms we associate to lattices, with the help of the following corollary that is a consequence of Proposition \ref{PosQ}.   

\begin{cor}\label{le:finQ} Let $Q_1$ and $Q_2$ be positive definite $n \times n$ matrices.  Let $\lambda_{\min}$ be the smallest eigenvalue of $Q_1$. If $B=[b_j]$ is a matrix with column vectors $b_j$, and $B^TQ_1B=Q_2$, then 
\[ b_j^TQ_1b_j=(Q_2)_{jj}, \quad j=1, \ldots, n.\] 
Moreover, $\|b_j\|^2\le (Q_2)_{jj}/\lambda_{\min}$ for each $j=1, \ldots, n$. 
\end{cor}

To lattices $A_i \Z^n$ for $i=1,2$ we associate the class of quadratic forms that are integrally equivalent to $Q_i = A_i ^T A_i$ for each $i$.  The lattices are congruent if and only if these equivalence classes are identical.  In turn, these equivalence classes are identical if and only if there is a unimodular matrix $B$ such that $B^TQ_1B=Q_2$.  Since the entries of unimodular matrices are integers, there are only finitely many unimodular matrices whose column vectors satisfy the conditions of Lemma \ref{le:finQ}.  Consequently, one can write a computer program to check congruency based on Corollary \ref{le:finQ}.  A greater challenge is to investigate infinite families of quadratic forms.  In 2011, Cervino and Hein \cite{hein} developed a systematic method to analyze isometry and non-isometry of infinite families of pairs of quadratic forms using modular forms.  An entirely different approach based on nodal counts was given one year later in \cite{bruning2012nodal}.

\subsection{Modular forms in connection to isospectrality}\label{ss:mod} 
Verifying with complete confidence that two flat tori are isospectral requires checking that all of their eigenvalues are identical counting multiplicity, yet again an apparently infinite task.  It turns out that the theory of \textit{modular forms} gives a convenient criterion for determining precisely when flat tori are isospectral. This is one of several reasons why the theory and language of modular forms appear in many of the articles related to the spectral theory of flat tori.

To state the definition, recall that $\mathrm{SL}_2(\Z)$ is the \textit{special linear group} of $2\times 2$ integer matrices with determinant $1$. The \textit{congruence groups} are defined as 
\[ \G_0(N):=\Big\{\begin{bmatrix}a & b \\ c& d\end{bmatrix}\in \GL_2(\Z): c\equiv 0 \mod N\Big\}.\] 
Let $\mathbb{H}$ denote the complex upper half-plane.  

\begin{defn}[Dirichlet characters and modular forms] A Dirichlet character mod $N$ is a function $\chi: \Z \to \C$ such that $\chi(mn)=\chi(m)\chi(n)$, $\chi(m+N) = \chi(m)$, $\chi(a)=0$ if $\gcd(a,N) >1$, and $\chi(a) \neq 0$ if $\gcd(a, N) =1$.  A Dirichlet character defines a character on the congruence group $\G_0(N)$ via 
\[ \chi(\gamma) := \chi(d), \quad \gamma = \begin{bmatrix} a&b \\ c& d \end{bmatrix} \in \G_0(N).\] 
Given a congruence group $\Gamma_0(N)$, a \textit{modular form} of weight $k$ and character $\chi$ mod $N$ is a function $f:\mathbb{H}\to \mathbb{C}$ such that:
\begin{enumerate}
    \item $f$ is holomorphic. 

    \item For any $\left[\begin{smallmatrix}a & b \\ c & d\end{smallmatrix}\right]\in \Gamma_0(N)$, $f(\frac{az+b}{cz+d})=\chi(d)(cz+d)^kf(z)$.
    
    \item For any $\left[\begin{smallmatrix}a & b \\ c & d\end{smallmatrix}\right]\in \SL_2(\Z)$, $(cz+d)^{-k}f(\frac{az+b}{cz+d})$ is bounded as $\mathrm{Im}(z)\to \infty$.
\end{enumerate}
The set of such forms is denoted  $M_k(\G_0(N),\chi)$, following the notation of \cite[p. 127]{koblitz2012introduction}.  
\end{defn}

In fact, it is enough to check conditions (2) and (3) for the generators of $\G_0(N)$ and $\mathrm{SL}_2(\ZZ)$ respectively.  One set of generators for $\SL_2(\Z)$ consists of the matrices $\left[\begin{smallmatrix}0 & -1\\ 1 & 0\end{smallmatrix}\right]$ and $\left[\begin{smallmatrix}1 & 1\\ 0 & 1\end{smallmatrix}\right]$.

Let us now restrict to \textit{rational} quadratic forms (respectively rational flat tori and rational lattices).  Up to a constant, any rational quadratic form is an even quadratic form, as in Definition \ref{def:sets_of_forms}. In the sense of Proposition \ref{prop: int tuple}, this restriction does not limit the investigation of the relationship between isospectrality and congruence, because this can be reduced to only considering rational quadratic forms.  One connection between modular forms and the spectrum of flat tori is through the following theorem that shows that the theta series of an even-dimensional rational lattice is a modular form.

\begin{theorem}[see p. 295 of  \cite{eichler1963einfuhrung}, or Corollary 4.9.5 (iii) of \cite{miyake2006modular}] \label{thm:eichler} Let $Q$ be an even positive definite quadratic form of $2k$ variables, and $N_Q$ be the smallest positive integer such that $N_Q Q^{-1}$ is even. If $\G$ is an underlying lattice, then $\theta_{\G}(z)\in M_k(\G_0(N_Q),\chi)$. The character $\chi$ mod $N_Q$ is defined for 
\begin{equation*} \label{eq:char_leg} \begin{gathered} \gamma = \begin{bmatrix} a&b \\ c& d \end{bmatrix} \in \G_0(N_Q), \quad  \chi(\gamma) :=  \\ 
 \begin{cases} \sgn(d)^k & (-1)^k \det(Q) \not \equiv 0 \textrm{ mod $|d|$ and is a quadratic residue mod $|d|$,} \\ -\sgn(d)^k & (-1)^k \det(Q) \textrm{ is a non-quadratic residue mod $|d|$,} \\ 0 & (-1)^k \det(Q) \equiv 0 \textrm{ mod } |d|. \end{cases} \end{gathered} 
\end{equation*}  
\end{theorem}

One can use the dual quadratic form of $q$ to show that $N_Q$ is a spectral invariant of even positive definite quadratic forms and their underlying lattices.  Modular forms in the same space $M_k(\G_0(N_Q),\chi)$ are determined by finitely many coefficients in their Fourier expansions.  Moreover, this number of coefficients is estimated from above by the \em Sturm bound \em  \cite{sturm_bound}. 

\begin{theorem}[Hecke's identity theorem for modular forms {\cite[p. 811]{heckemathematische}}]\label{th:id}
Let 
\[ f(z)= \sum_{n=0}^\infty a_ne^{2\pi i nz}\in M_k(\G_0(N),\chi), \quad \mu_0(N):=N\prod_{p | N, \, \textrm{ prime} } \left(1+\frac{1}{p}\right).\]  Assume that $\chi$ is real-valued.  
Then the first $\frac{\mu_0(N)k}{12} + 1$ coefficients $a_n$ completely determine $f$.  Equivalently, 
\[ a_n=0 \, \forall n \textrm{ with } 0\leq n \leq \frac{\mu_0(N)k}{12}+1 \implies f=0.\]
\end{theorem}

We first apply the identity theorem to modular forms in even dimensions. 

\begin{cor}[Isospectrality certificate]\label{cor: modisos} Let $P$ and $Q$ be two even positive definite quadratic forms of $2k$ variables. They are isospectral if and only if $\det(P)=\det(Q)$, $N_P=N_Q$, and their multiplicities over the integers of values less than or equal to $\frac{\mu_0(N_P)k}{12}+1$ coincide.
\end{cor}

\begin{proof} It is straightforward to verify that if $P$ and $Q$ are isospectral, then the statements of the corollary hold.  For the converse, note that for any underlying lattice of $Q$, denoted $A \Z^n$, the theta series of this lattice is 
\[\sum_{\gamma \in A \Z^n} e^{i\pi z ||\gamma||^2} = \sum_{x \in \Z^n} e^{i \pi z (Ax)^T (Ax)} = \sum_{x \in \Z^n} e^{i \pi z x^T Q x}.\] 
Consequently, the theta series is identical for all underlying lattices, and we can therefore define the theta series associated to $P$ and $Q$  
\[ \theta_P(z) := \sum_{x \in \Z^n} e^{i \pi z x^T P x}, \quad \theta_Q (z) := \sum_{x \in \Z^n} e^{i \pi z x^T Q x}.\] 
If $N_P=N_Q$, and $\det(P) = \det(Q)$, then their corresponding theta series $\theta_P,\theta_Q$ lie in $M_k(\G_0(N_P),\chi)$ for some $\chi$ determined by $\det(P)=\det(Q)$, by Theorem \ref{thm:eichler}.  The quadratic forms are isospectral if and only if their theta series are identical.  So, consider $f(z) :=\theta_P(z) -\theta_Q(z)$.
We have 
\begin{align*}
    M_k(\G_0(N_P),\chi)\ni f(z)&=\sum_{x\in \Z^n} e^{\pi i zx^TPx} -\sum_{x\in \Z^n} e^{\pi i zx^TQx}\\
    &= \sum_{n\in \NN}  \big(m_P(n)-m_Q(n)\big)e^{\pi izn},
\end{align*}
where $m_P(n)$, $m_Q(n)$ respectively denote the multiplicities of the value $n$ of $P$ and $Q$ as quadratic forms over $\ZZ^n$. By Theorem \ref{th:id},  if $m_P(n)-m_Q(n)=0$ for all $0 \leq n \leq \mu_0(N_P)k/12+1$, then $f(z)=0$ for all $z$, and $\theta_P = \theta_Q$.  
\end{proof}

To extend Corollary \ref{cor: modisos} to odd dimensions, note that if $P$, $Q$ are positive definite $k \times k$ matrices, they define isospectral quadratic forms if and only if the $k+1$-dimensional forms defined by 
$$P'=\begin{bmatrix}
2 & 0 \\
0 & P
\end{bmatrix}\quad \& \quad Q'=\begin{bmatrix}
2 & 0 \\
0 & Q
\end{bmatrix} $$
are isospectral. Moreover, $P'$ and $Q'$ are integrally equivalent if and only if $P$ and $Q$ are.

\subsection{Building lattices from linear codes} \label{ss:lincodes} 
Codes are used in numerous everyday circumstances including data compression, cryptography, error detection and correction, data transmission, and data storage.  Linear codes are useful for studying lattices because they allow one to translate questions for lattices, infinite discrete groups, into questions for finite groups, known as linear codes.  For a general treatment we refer to \cite{ebeling2013lattices} and \cite{nebe2006self}.

\begin{defn}
    A \textit{linear $q$-nary code $C$ of length $n$} is a $\ZZ$-linear subspace (and a subgroup) of the module $(\Z/q\Z)^n$, where $\Z/q\Z$ is the ring of integers modulo $q$. Its elements are \textit{codewords}.
    The linear space $(\ZZ/q\ZZ)^n$ is equipped with the inner product $x\cdot y:= \sum_{i=1} ^n x_iy_i$ $\mathrm{mod}$ $q$. 
\end{defn}

In the literature, $q$ is often assumed to be prime, but we don't need this assumption for our purposes. To construct a lattice from a linear code consider the projection 
\begin{align*}
    \pi:  \Z^n \to (\Z/q\Z)^n \quad\textnormal{by}\quad  z  \mapsto z\  \text{mod }q,
\end{align*}
where mod $q$ acts coordinate-wise. Importantly, $\pi$ is a group homomorphism. 

Let $C$ be a linear code in $(\ZZ/q\ZZ)^n$. The pre-image under $\pi_q$, 
\[  \inv{\pi}_q(C) := \{ \ell\in \ZZ^n : \ell \ \text{mod }q \in C \} \]
is a full-rank lattice; see \cite[Prop. 16.2]{kacroot}.  The codewords $c_i\in C$ partition $\pi_q^{-1}(C)$ into the subsets $\pi_q^{-1}(c_i)$. So, if  $C_1,C_2\subseteq (\ZZ/q\ZZ)^n$ are different codes, then $\pi_q^{-1}(C_1)\neq\pi_q^{-1}(C_2)$.  If $c_1,\ldots,c_k$ are generators of $C$, then 
\begin{equation} \label{eq:inv_code_latt} \pi_q^{-1}(C)=\begin{bmatrix}c_1\: \cdots \:c_k\: qI\end{bmatrix}\ZZ^{k+n}\subseteq \ZZ^n, \end{equation}
with $I$ the $n\times n$ identity matrix.  In \cite{conwaysloane}, building a lattice in this way is called \em construction A. \em In the following theorem, we prove that \em all \em integer lattices can be obtained by construction A, indicating the usefulness of this approach. As a consequence of this characterization, there are only finitely many distinct integer lattices of a given determinant $q$, because the number of linear codes in $(\ZZ/q\ZZ)^n$ is finite.

\begin{theorem} Any integer lattice $L$ is the pre-image of the $\mathrm{vol}(L)$-nary code $\pi_{\mathrm{vol}(L)}(L)$.
\end{theorem}

\begin{proof} Assume that $L$ is an integer lattice.  For any $q$, define 
 \[ L':=\inv{\pi_{q}}(\pi_{q}(L)) = \{x\in \ZZ^n:\pi_q(x)\in \pi_q(L) \}.\]  
For $x \in \Z^n$, $\pi_q(x)\in \pi_q(L)$ if and only if there is some $\g\in L$ such that $\pi_q(x)=\pi_q(\g)$. This is equivalent to $\pi_q(x-\g)=0$, which holds if and only if $x=\g+qz$ for some $z\in \ZZ^n$. This proves $L'=L+q\ZZ^n$. For $q=\mathrm{vol}(\G)$, $q\ZZ^n\subseteq L$ by Proposition \ref{prop:sublattice}, and therefore $L'=L$.
\end{proof}

\begin{cor}\label{cor:eqlattice} Let $C\subseteq (\ZZ/q\ZZ)^n$ be a linear code and $L=A\ZZ^n$ be a lattice, where $A=[a_j]$. If the codewords $\pi_q(a_j)$ generate $C$, then $L\subseteq \pi_q^{-1}(C)$. If in addition, $q\ZZ^n\subseteq L$, then $L=\pi_q^{-1}(C)$. 
\end{cor}

\begin{proof} If $\pi_q(a_j)$ generate $C$, then $\pi_q(L)=C$, so $L \subseteq \pi_q ^{-1} (\pi_q (L)) = \pi_q^{-1} (C)$.
If $q\ZZ^n\subseteq L$, then the inclusion becomes an equality. 
\end{proof}

We next relate linear codes to isospectrality. For this we need 

\begin{defn} Let $C_1,C_2\subseteq (\ZZ/q\ZZ)^n$ be two linear codes of equal cardinality and list their respective elements as $c_1^{(i)},\ldots,c_k^{(i)}$ for $i=1,2$. The codes have the same \textit{weight distribution} if for each pair $(c_j^{(1)},c_j^{(2)})$, there is a permutation $\sigma \in S_n$ such that $(c_j^{(2)})_{k}=(c_j^{(1)})_{\sigma(k)}$ for each coordinate $k$. The codes constitute an \textit{absolute pairing} if for each $(c_j^{(1)},c_j^{(2)})$ we have $(c_j^{(2)})_k=\pm (c_j^{(1)})_k$ in $\ZZ/q\ZZ$ for each $k$. Both relations are equivalence relations. 
\end{defn}

\begin{prop}\label{prop: weight} Let $C_1,C_2$ be $q$-nary linear codes, and let $L_i=\inv{\pi}_q(C_i)$. If the weight distributions of $C_1$ and $C_2$ are the same, then $L_1$ and $L_2$ are isospectral.  The converse does not hold.
\end{prop}

\begin{proof} We prove the first part.  Let $c^{(i)}_1,\ldots,c^{(i)}_m$ be lists of the codewords as in the definition of equal weight distribution. We give length-preserving bijections from $\pi_q^{-1}(c_j^{(1)})$ to $\pi_q^{-1}(c_j^{(2)})$ for each $j$ which is sufficient since the inverse images of the codewords partition $\pi_q^{-1}(C_i)$.  Let $\sigma$ be the permutation corresponding to the pair $(c_j^{(1)},c_j^{(2)})$. It can be realized as a permutation matrix $\Sigma$, which is orthogonal. Then $x\mapsto \Sigma x$ is length-preserving and by construction a bijection from $\pi_q^{-1}(c_j^{(1)})$ to $\pi_q^{-1}(c_j^{(2)})$. Consequently, $L_i$ have identical length spectra and are therefore isospectral. 
\end{proof}

\begin{theorem}\label{th: absChar} Let $C_1,C_2\subseteq (\Z/q\Z)^n$ be linear codes with 
\[ L_i=A_i\ZZ^n=\inv{\pi}_q(C_i).\] 
There is a bijection $\phi:L_1 \to L_2$ preserving the absolute value of each coordinate if and only if $C_1$, $C_2$ make an absolute pairing. If either is true, then $DL_1$ and $DL_2$ are isospectral for any diagonal invertible matrix $D$.
\end{theorem}

\begin{proof}[Proof of Theorem \ref{th: absChar}] $ $ 
$\Leftarrow)$ Let $c^{(i)}_1,\ldots,c^{(i)}_m$ be lists of codewords as in the definition of an absolute pairing, and view them as elements of $\{0,1,\ldots,q-1\}^n\subseteq \ZZ^n$. We again give a length-preserving bijection $\phi:\pi_q^{-1}(c_j^{(1)})\to\pi_q^{-1}(c_j^{(2)})$ for a fixed $j$. Write $x\in \pi_q^{-1}(c_j^{(1)})$ uniquely as $c_j^{(1)}+qt$, where $t\in \ZZ^n$. Let $\phi$ act coordinate-wise bijectively as follows:
\[ x_k=(c_j^{(1)})_k+qt_k\mapsto \begin{cases}(c_j^{(2)})_k+qt_k,& \textnormal{ 
if }(c_j^{(2)})_k=(c_j^{(1)})_k\\
(c_j^{(2)})_k-q(t_k+1), & \textnormal{ 
otherwise if }(c_j^{(2)})_k=q-(c_j^{(1)})_k.
\end{cases}
\] 
This map preserves the absolute values of the coordinates.  For the last part, consider $\phi:L_1\to L_2$, a bijection between lattices and a diagonal matrix $D$ as in the statement above. By assumption, if $x\in L_1$, then $|x_i|=|\phi(x)_i|$ for each coordinate. Now define $\varphi:DL_1\to DL_2$ as $\varphi(Dx):=D\phi(x)$ for $x\in L_1$. We have for $x\in L_1$, $|(Dx)_i|=|d_ix_i|=|d_i\phi(x)_i|=|(D\phi(x))_i|=|\varphi(Dx)_i|$. Therefore $\varphi$ is again a bijection that preserves the absolute values of coordinates, from whence it follows that $DL_1$ and $DL_2$ are isospectral.  We leave the direction $\Rightarrow$ of the proof to the reader. 
\end{proof}

Isomorphic codes (in the sense of groups) do not in general correspond to either isospectral or congruent lattices.

\subsubsection{Can one hear cubicity?} \label{ss: rect6} 
In \S \ref{s:rectangular}, we proved that if a pair of rectangular flat tori are isospectral, then they are isometric.  Suppose we only know a priori that one flat torus in the pair is rectangular, then surely the other must be as well? As an application of linear codes, we show that the answer depends on the dimension.  Following Conway \cite[p. 40--42]{conway}, we define \em cubic lattices \em as those that are congruent up to scaling to $ \Z^n$ and say that \em cubicity \em is the property of being cubic.

\begin{prop}
\label{prop: 6dim} The two lattices
\[\Lambda=\left[\begin{smallmatrix}
1 & 1 & 0 & 0 & 0 & 0\\
1 & -1 & 0 & 0 & 0 & 0\\
0 & 0 & 1 & 1 & 0 & 0\\
0 & 0 & 1 & -1 & 0 & 0\\
0 & 0 & 0 & 0 & 1 & 1\\
0 & 0 & 0 & 0 & 1 & -1
\end{smallmatrix}\right]\ZZ^6 \quad \& \quad \Omega=\left[\begin{smallmatrix}
1 & 1 & 0 & 0 & 0 & 1\\
0 & 0 & 0 & 0 & 1 & 1\\
1 & -1 & 1 & 0 & 0 & 0\\
0 & 0 & 0 & 0 & 1 & -1\\
0 & 0 & 1 & 0 & 1 & 0\\
0 & 0 & 0 & 2 & 1 & 1\end{smallmatrix}\right]\ZZ^6,\]
are isospectral and non-congruent. In particular, $\Lambda$ is cubic, but $\Omega$ is not.  Cubicity is audible in dimensions five and lower; it is not audible in dimensions $6$ and higher.   
\end{prop}

\begin{proof} We prove that the pair is isospectral and non-congruent.  As a consequence, Lemma \ref{le:schiemlemma} shows that cubicity is not audible for $n>6$. The fact that cubicity is audible when $n<6$ is shown in \cite[p. 60]{conway}. Consider the following two binary linear codes of length 6.  The rows are the codewords, and the codewords of the same row differ by permutation, showing that the codes are of the same weight distribution: 
\[ C_1: \begin{smallmatrix}
0 & 0 & 0 & 0 & 0 & 0\\
1 & 1 & 0 & 0 & 0 & 0\\
0 & 0 & 1 & 1 & 0 & 0\\
0 & 0 & 0 & 0 & 1 & 1\\
1 & 1 & 1 & 1 & 0 & 0\\
1 & 1 & 0 & 0 & 1 & 1\\
0 & 0 & 1 & 1 & 1 & 1\\
1 & 1 & 1 & 1 & 1 & 1
\end{smallmatrix}\quad \& \quad C_2:\begin{smallmatrix}
0 & 0 & 0 & 0 & 0 & 0\\
1 & 0 & 1 & 0 & 0 & 0\\
0 & 0 & 1 & 0 & 1 & 0\\
1 & 0 & 0 & 0 & 1 & 0\\
0 & 1 & 0 & 1 & 1 & 1\\
1 & 1 & 0 & 1 & 0 & 1\\
0 & 1 & 1 & 1 & 0 & 1\\
1 & 1 & 1 & 1 & 1 & 1
\end{smallmatrix}. \] 
We have that, modulo 2, the vectors of the given bases of $\Lambda$, $\Omega$ generate the codes $C_1$, $C_2$ respectively.
It is straightforward to check that, $2\ZZ^6\subseteq \Lambda$, $\Omega$.  So, the pre-images of the codes are equal to $\Lambda$, $\Omega$ respectively by Corollary \ref{cor:eqlattice}. They are isospectral by Proposition \ref{prop: weight} and non-isometric since the vectors of length $\sqrt 2$ 
in $\Lambda$ that are non-parallel are orthogonal, while this is not the case for $\Omega$.
\end{proof}


\section{Duets of isospectral non-isometric flat tori} \label{s:differentdonuts} 
In 1964, a paper of one single page was published in Proceedings of the National Academy of Science that is famous to this day \cite{milnor}, \em Eigenvalues of the Laplace operator on certain manifolds. \em In it Milnor described an example of two sixteen dimensional flat tori that are isospectral but not isometric.  One can imagine this as a \em duet, \em a pair of perfectly attuned yet differently shaped flat tori that would resonate with identical frequencies.  Milnor's paper  inspired Kac's acclaimed work \cite{kac} titled \em Can one hear the shape of a drum? \em  The resolution of Kac's question by Gordon, Webb, and Wolpert \cite{gww} was based in part on adapting the  \em Sunada method \em \cite{sunada} from four dimensions to two.  Sunada humbly described this method as \cite[p. 169]{sunada} ``a \em geometric \em analogue of a routine method in number theory.''   This simple and elegant method constructs not only isospectral Riemannian manifolds but also families of isospectral abelian varieties.

\subsection{Milnor's duet} \label{ss:milnor} 
Milnor's paper \cite{milnor} referred to a construction of two lattices by Witt \cite{witt} that begins with the 
root lattice $D_n$, also called the \textit{checkerboard lattice}, 
\[ D_n := \left\{ z = (z_1, \ldots, z_n) \in \Z^n : \sum_{i=1} ^n z_i \in 2\ZZ \right\}.\]
A basis for $D_n$ is given by the vectors $\{ e_1+e_2, e_{j-1} - e_j \}_{j=2} ^n$. 
The root lattice of the $E_n$ root system, also denoted by $E_n$, for $n$ divisible by 4 is 
\[E_n :=\left\{x\in \ZZ^n\cup\left(\frac{1}{2}\mathds{1}+\ZZ^n\right): \sum x_i\in 2\ZZ \right\}, \quad \mathds{1} := \sum_{j=1} ^n e_j.\]
We note that $E_8$ is sometimes known as the Gosset lattice after \cite{gosset}.
 
When $n$ is divisible by 4, a basis for $E_n$ is given by $ \left\{e_1+e_2,e_{j-1}-e_{j}, \frac{1}{2}\mathds{1}\right\}_{j=2}^{n-1}$. 
The classical theory of root lattices tells us that $D_n$ with $n>2$ and $E_{4n}$ are irreducible lattices; see \cite[\S 1.4]{ebeling2013lattices} and \cite[\S 4.7-4.8.1]{conway2013sphere}.  We give an alternative method for checking irreducibility that may be of independent interest as we have not seen this elsewhere in the literature.  This method can be applied to prove irreducibility of $D_n$ and $E_{4n}$ but may also be more broadly applicable. 

\begin{lemma} \label{co: irredSpecial}Let $A=[a_j]$ be a basis of a lattice, $\G\subseteq \RR^n$ with irreducible components $\G_i$. For a scalar $s\neq 0$, and a vector $v \in \R^n$, the lattice 
\[ \Lambda: =\begin{bmatrix}
A & v \\
0 & s
\end{bmatrix}\ZZ^{n+1}\]  
is irreducible if:
\begin{enumerate}[label=(\alph*)]
    \item each $(a_j,0)$ is of shortest non-zero length in $\Lambda$,
    \item $(v+\g) \cdot \G_i \neq \{0\}$ for each $\g \in \G$ and $i$. 
\end{enumerate}
\end{lemma}
\begin{proof}
By contradiction, assume that we can write $\Lambda=\Lambda_1\oplus\Lambda_2$, where $\Lambda_i$ are non-trivial, and $\Lambda_1$ is irreducible. Without loss of generality, since $s\neq 0$, we may assume that there is some $\omega \in \Lambda_1$ with $\omega = (x, t)$ for some $t \neq 0$ (for simplicity we use row vector notation in this proof). Then, $\{\omega, (a_j, 0)\}_{j=1} ^n$ are all contained in $\Lambda$, and they are all linearly independent.  Each $(a_j,0)$ lies in precisely one of $\Lambda_i$, by virtue of (a) and the Pythagorean theorem.  Consequently, the number of $(a_j,0)$ that lie in $\Lambda_2$ is equal to the rank of $\Lambda_2$.  In particular, the last coordinate of every element of $\Lambda_2$ vanishes. Therefore, the projection of $\Lambda_2$ onto the first $n$ coordinates is an orthogonal sum of some $\G_i$. We fix an element $\G_i$ in this sum. Since $(v,s) \in \Lambda$, $(v,s)$ is equal to a sum of an element of $\Lambda_1$ and and an element of $\Lambda_2$. We can therefore without loss of generality choose $\omega=(x,s) \in \Lambda_1$ such that $\omega=(v,s)  + (\gamma, 0)$ with $(\gamma,0) \in \Lambda_2$, and $\gamma \in \G$.  Since $\omega \in \Lambda_1$, $(v+\gamma) \cdot \Gamma_i$ vanishes, a contradiction to (b).
\end{proof}

Since $E_{16}$ is irreducible, it is not congruent to $E_8\times E_8$.  Milnor's result is completed by showing that they are isospectral.  These lattices are even and unimodular, and consequently their theta series are modular forms for $\PSL_2(\Z)$.  In 16 dimensions there is only one such form (up to multiplication by scalars), so since these lattices have identical volume, they also have identical theta series and are therefore isospectral.

\subsection{The race to find duets}\label{ss:race} 
As we have seen, one dimensional isospectral flat tori are always isometric, so a natural question is, what is the lowest dimension in which there are isospectral non-isometric flat tori?  Following Milnor, the search for isospectral and non-isometric duets of flat tori became a race towards the lowest possible dimension.  Kneser found a 12 dimensional example \cite{kneser} in 1967.  Ten years later, Kitaoka \cite{kitaoka} reduced this to 8.  In 1986, Conway and Sloane \cite{conwaysloane} found 5 and 6 dimensional examples.  In 1990, Schiemann\cite{schiemann1} constructed a 4 dimensional example. Independently, and using a different approach, Shiota \cite{shiota1991theta} found another example one year later in 1991. The same year, Earnest and Nipp \cite{earnest1991theta} contributed with one more pair.

Kneser's $12$-dimensional pair is $D_{12}$ and $E_8\times D_4$ \cite{kneser}. Kitaoka also made use of $D_4$ in his construction \cite{kitaoka}.  We refer interested readers to the literature for the aforementioned constructions and recall here Schiemann's four dimensional pair.  Consider the positive definite matrices
$$\begin{bmatrix} 4 & 2 & 0 & 1\\
2 & 8 & 3 & 1\\
0 & 3 & 10 & 5\\
1& 1 & 5 & 10
\end{bmatrix}\quad \&\quad \begin{bmatrix} 4 & 0 & 1 & 1\\
0 & 8 & 1 & -4\\
1 & 1 & 8 & 2 \\1 & -4 & 2 & 10
\end{bmatrix}. $$
Schiemann proved that the quadratic forms they define have identical representation numbers using Corollary \ref{cor: modisos}.  He showed that these forms are not integrally equivalent using the theory of Minkowski reduction; see \ref{ss:Mink}.  This can also be seen by Lemma \ref{le:finQ}. These two quadratic forms were systematically found in the sense that they are the integral forms that satisfy these conditions and have the smallest determinant.

\subsection{Conway and Sloane's isospectral family}\label{ss:inf} 
One way to obtain a family of infinitely many pairs of isospectral non-isometric 4-dimensional flat tori is to start with Schiemann's pair, denoted by $(S_1, S_2)$ and take the one parameter family $(cS_1, cS_2)_{c >0}$.  We say that such a family is obtained by scaling.  
Can one obtain an infinite family of isospectral non-isometric pairs that are not simply obtained by scaling?  Conway and Sloane were the first to present an infinite family of 4-dimensional pairs via 

\begin{theorem}[Conway-Sloane {\cite{conwaysloane}}, and Cervino-Hein {\cite{hein}}] \label{th:ConSlo}
Let $a,b,c,d>0$. Consider the two matrices
\[ A_{\pm}=\frac{1}{\sqrt{12}}\left[\begin{smallmatrix}
\sqrt{a} & 0 & 0 & 0\\
0 & \sqrt{b} & 0 & 0\\
0 & 0 & \sqrt{c} & 0\\
0 & 0 & 0 & \sqrt{d}
\end{smallmatrix}\right]\underbrace{\left[\begin{smallmatrix}
\pm3 & 1 & 1 & 1\\
-1 & \pm3 & -1 & 1\\
-1 & 1 & \pm3 & -1\\
-1 & -1 & 1 & \pm3
\end{smallmatrix}\right]}_{=:T_{\pm}}.  \] 
There is a bijection $T_+\ZZ^4\to T_-\ZZ^4$ preserving absolute values of the coordinates, and therefore the lattices $A_+\ZZ^4$, $A_-\ZZ^4$ are isospectral for any $a,b,c,d>0$. They are non-isometric if and only if  $a,b,c$, and $d$ are all distinct.
\end{theorem}

Both $A_{\pm}$ have determinant $\sqrt{abcd}$. The observant reader may see a connection to Theorem \ref{th: absChar}. Indeed, $T_{\pm}\ZZ^4$ correspond to linear codes that constitute an absolute pairing. However, since $\det(T_{\pm})=144$, this is not easy to check directly. This theorem encompasses, up integral equivalence, Schiemann's pair $(S_1,S_2)$ of quadratic forms by letting $a=1,b=7,c=13,d=19$, which can be seen using Lemma \ref{le:finQ}. Conway and Sloane were only able to verify non-isometry for integers $a<b<c<d$ with $abcd<10,000$ \cite{conwaysloane}; for other values of $a, b, c, d$ the non-isometry was a conjecture at that time.   
Their proof of isospectrality \cite{conwaysloane} is short and simple, but the proof of the non-isometry conjecture resisted solution for nearly 20 years.  Its proof was completed by Cervino and Hein \cite{hein} in 2011 and is significantly more complicated than the proof of isospectrality.  Conway, Sloane, Cervino, and Hein therefore together obtained the following corollary.  

\begin{cor} \label{prop:infex} There are infinitely many pairs of isospectral non-isometric four-dimensional flat tori, that are not simply related by scaling.
\end{cor}

\section{The race ends to the sound of Schiemann's symphony} \label{s:schiemann}
The race to find isospectral non-isometric flat tori in successively lower dimensions finally ended in 1994, over three decades after it began, when Alexander Schiemann determined the final answer to the three equivalent questions in \S \ref{ss:three}. The aim of this section is to describe \em Schiemann's symphony\em, the elaborate algorithm that gave a final answer to this question, and that was performed with a comprehensive computer search \cite{schiemann1994ternare}, \cite{schiemann2}.


\subsection{Schiemann's theorem} 
To quantify the isospectral question for flat tori, we introduce the \em choir numbers. \em  
\begin{defn}\label{Fchoir} In each dimension $n\in \N$, we define the \em choir number \em $\flat_n$ to be the maximal number $k$ such that the three following equivalent conditions hold:\footnote{We use the musical symbol $\flat$, ``flat'' since we are working with flat tori.} there is a sequence of $k$ $n$-dimensional mutually isospectral 
\begin{enumerate}
   \item flat tori that are non-isometric,
    \item positive definite quadratic forms that are non-integrally equivalent,
    \item lattices that are non-congruent. 
\end{enumerate}
\end{defn}

Since isospectral flat tori would sound identical, the choir numbers describe how many geometrically distinct flat tori can resonate in perfect unison.  In 1990--1994, Schiemann proved the next remarkable theorem \cite{schiemann1}, \cite{schiemann1994ternare}.

\begin{theorem}[Schiemann's theorem]\label{th:Schi} The choir numbers are equal to $1$ for $n=1,2,3$ and greater than $1$ otherwise.
\end{theorem}

\begin{proof}[Proof that $\flat_2=1$] 
In dimension 2, a full-rank lattice always has a basis consisting of a shortest non-zero vector, and a shortest vector that is linearly independent of the first one.  Interestingly, it is \em not \em always possible to find such a basis for higher dimensional lattices; see \S\ref{ss:Mink}.  
Let $A,A'$ be bases of two isospectral lattices $\G,\G'$ in $\RR^2$ with the aforementioned property. By left-multiplication with an orthogonal matrix, we may assume that 
$$A=\begin{bmatrix}
a & c\\
0 & b
\end{bmatrix}\quad \&\quad
A'=\begin{bmatrix}
a' & c' \\
0 & b'
\end{bmatrix}.$$
Here, $a,a',b,b'$ are positive numbers. By Corollary \ref{cor:dim_vol}, $a=a'$, and $b=b'$. If $|c|< |c'|$, then the closed ball $\overline{D(0, \|(c,b)\|)}$ contains more points from $\G$ than from $\G'$, implying that $\G,\G'$ are not isospectral. If $c=-c'$, then by letting $B=C=\left[\begin{smallmatrix}-1 & 0 \\ 0& 1\end{smallmatrix}\right]$, we see that $CA'=AB$ and $\G$ is congruent to $\G'$. 
\end{proof}

The proof that $\flat_1=1$ was shown in the introduction, and the proof that $\flat_2=1$ above is also quite short.  In contrast, proving $\flat_3=1$ is so difficult that to the best of our knowledge, it has never been achieved without the help of a computer.  This proof in full detail is only available in German \cite{schiemann1}. Consequently, we take this opportunity to present the main ideas, general structure, and strategy of the proof.  This strategy is independently interesting, and we further suspect that Schiemann's methods can be generalized to higher dimensions in order to determine, say $\flat_4$, which is unknown. We describe the proof of Theorem \ref{th:Schi} as \em Schiemann's symphony. \em This piece is performed using positive definite quadratic forms. Up to integral equivalence, these forms can be geometrically represented as a polyhedral cone, which is a convenient structure for our computer algorithms.  The proof is rather technical, so we employ musical analogies as a mnemonic technique to keep track of the different elements in the proof and the roles they play. 



\subsection{Minkowski reduction}\label{ss:Mink} 
A \em representative \em of a quadratic form $Q$ is any element of the equivalence class of integrally equivalent quadratic forms containing $Q$.  There are infinitely many representatives in each integral equivalence class of positive definite quadratic forms. A \em Minkowski reduced \em form is a particularly natural representative which has been of historical interest as described in Sch\"urman's survey \cite{schurmann2009computational}.

\begin{defn}[Minkowski reduced forms and bases] \label{def:min_red} A positive definite quadratic form $q$ is \em Minkowski reduced \em if for all $k = 1, \ldots ,n$ and for all $x \in \Z^n$ with $\gcd(x_k, \ldots, x_n) = 1$, we have $q(x)\geq q_{kk}$. Moreover, a lattice basis $A=[a_j]\in \mathrm{GL}_n(\RR)$ is \textit{Minkowski reduced} if for each $j$,
\[ a_j\in A\Z^n\setminus \{0,a_1,\ldots,a_{j-1}\} , \] 
is a shortest choice of vector such that $a_1,\ldots,a_j$ is part of some basis of $\G$. 
\end{defn}

We summarize the key properties of Minkowski reduced forms and bases in the following proposition.  
\begin{prop} \label{prop:posDmink} (1) A quadratic form $q$ is positive definite and Minkowski reduced if and only if $q_{11}>0$ and $q(x)\geq q_{kk}$ for all $k=1,\ldots,n$ as long as $x\in \Z^n$ with $\gcd(x_k,\ldots,x_n)=1$.  (2) $A$ is a Minkowski reduced lattice basis if and only if $A^TA$ is a Minkowski reduced positive definite form. (3) Each positive definite quadratic form has a positive, finite number of Minkowski reduced representatives. 
\end{prop} 
\begin{proof} 
The first statement follows from Definition \ref{def:min_red}.  The second statement is demonstrated in \cite[Cor. 4, p. 14]{cassels2012introduction}.  The third statement is proven in \cite[p. 27-28]{cassels2012introduction}, alternatively  \cite[Section 4.4.2]{terras2012harmonic}.
\end{proof}

Is it possible to find an even more intuitive reduction? For example, is there always, given a lattice $\G$, a basis matrix $A=[a_j]$ such that each 
\begin{equation*} \label{Intu}
     a_j\in \G\setminus\Span \{0,a_1,\ldots,a_{j-1}\} 
\end{equation*}
is \textit{any} shortest choice of vector? 
In four dimensions, the following example constructed by van der Waerden \cite[p. 286]{van1956reduktionstheorie} shows that this is not always possible. Consider the basis matrix  
$$A_4=\left[\begin{smallmatrix}
1 & 0 & 0 & 1/2\\
0 & 1 & 0 & 1/2 \\
0 & 0 & 1 & 1/2 \\
0 & 0 & 0 & 1/2\\
\end{smallmatrix}\right] $$
of the lattice $\G=A_4\Z^4$, with the column vector notation $A_4=[a_j]$. We see that $e_4=-a_1-a_2-a_3+2a_4$ is of the same length as $a_4$ and is linearly independent of $a_1,a_2,a_3$, however the vectors $a_1,a_2,a_3,e_4$ do not make a basis for $\G$. In general for $n\geq 5$, let $A_n=[e_1, \ldots, e_{n-1},\frac{1}{2}\mathds{1}]$, where $\mathds{1}=e_1+\cdots+e_n$. Then $e_n\in A_n\Z^n$ is shorter than $\frac{1}{2}\mathds{1}$, but $e_1,\ldots,e_n$ is not a basis for $A_n\Z^n$.

\begin{theorem}[see {\cite[p. 278]{van1956reduktionstheorie}}]\label{IntuRed} As long as $n\leq 3$, we can find a basis matrix $A=[a_j]$ for any $n$-dimensional lattice $\G$ such that $a_1$ is a shortest non-zero vector of $\G$, and each $a_i$ is \em any \em shortest choice such that $a_1,\ldots,a_i$ are linearly independent. If $n=4$, then we have the equivalent statement if we replace \em any \em with \em some. \em  
\end{theorem}

\begin{defn}[Minkowski Domain] We define $\mathcal{M}_n$ to be the set of $n$-dimensional symmetric positive definite quadratic forms that are Minkowski reduced.  
\end{defn}

\subsubsection{Polyhedral cones} To give a geometric description of the Minkowski domain we define \textit{polyhedral cones}. Properties of these cones are used throughout Schiemann's symphony.  First note that we can naturally embed $n$-dimensional symmetric quadratic forms $q(x)=x^TQx$ where $Q=(q_{ij})_{ij}$ in $\R^{n(n+1)/2}$. There are many ways to perform such an embedding. In three dimensions perhaps the most common embedding is 
\[ q\mapsto (q_{11},q_{22},q_{33},q_{12},q_{13},q_{23}). \]

\begin{defn}[Polyhedral cone]\label{def:polycone}  Let $A,B$ be (possibly empty) finite sets of non-zero $n$-dimensional vectors. Then a set of the form 
\[ \Pc(A,B):=\left\{x\in\mathbb{R}^n: a\cdot x\geq 0, \quad b\cdot x>0 \textrm{ for each } a\in A, \textrm{ and } b\in B\right\} \] 
is a \em polyhedral cone. \em For $a \in A$, a set of the form $\{ x \in \R^n : a\cdot x=0\}$ is a \textit{supporting hyperplane,} similarly for $b \in B$. The dimension of $\Pc(A,B)$ is the dimension of the smallest vector space containing it.  A polyhedral cone is \textit{pointed} if it does not contain any lines.  
\end{defn}

It follows from the definition that the closure of a non-empty polyhedral cone 
\[ \overline{\Pc(A,B)} = \Pc\left(A\cup B,\emptyset\right).\] 
Rational polyhedral cones (i.e., the elements of $A$ and $B$ are vectors with rational entries) are amenable to computers because they are easily stored, and a computer can do exact calculations.  Polyhedral cones have facets, faces, edges and vertices.    

\begin{defn}[Facets, $j$-faces, edges, and vertices] \label{def:fjfev} Let $P$ be a $k$-dimensional polyhedral cone with closure $\overline{P}$. A \em facet \em of $P$  is a $k-1$-dimensional intersection of $\overline{P}$ with a collection of its supporting hyperplanes.  A \em$j$-face \em is similarly a $j$-dimensional intersection. A $1$-face is an \textit{edge} and a $0$-face is a \textit{vertex}.
\end{defn}

The following proposition collects properties of polyhedral cones; the proof is omitted as these properties can be readily checked using Definition \ref{def:polycone}; see also \cite[Theorem 3.10]{joswig2013polyhedral}.

\begin{prop} \label{prop:pcs} 
Let $\Pc(A,B)\subset \mathbb{R}^n$ and $\Pc(C,D)\subset \mathbb{R}^m$ be polyhedral cones.  
\begin{enumerate} 
\item The Cartesian product is also a polyhedral cone, specifically 
\[ \Pc(A,B)\times \Pc(C,D) = \Pc \left(A\times \{0\}\cup \{0\}\times C,B\times \{0\}\cup \{0\}\times D \right). \] 
\item If $n=m$, then $\Pc(A,B)\cap\Pc(C,D) = \Pc\left(A\cup C,B\cup D\right)$.  
\item If $U$ is a closed set in $\R^n$, and $\Pc(A,B) \neq \emptyset$, then $\Pc(A,B)\subseteq U$ if and only if $\Pc\left(A\cup B,\emptyset\right)\subseteq U$.  
\item If the polyhedral cone $\Pc(A,B)$ is pointed, then it has finitely many edges that are sets of the form $k_i \R_{\geq 0}$, for $k_i \in \R^n$, and 
\begin{equation} \label{eq:pcone_edges} \overline{\Pc(A,B)}=\sum_{i=1}^rk_i\R_{\ge 0}.  \end{equation} 
Consequently, the dimension of $\Pc(A,B)$ is the dimension of the span of $\{k_i\}_{i=1} ^r$.  
\item If $k_i \R_{\geq 0}$ are the edges of $\Pc(A,\emptyset)$ and $k_j' \R_{\geq 0}$ are those of $\Pc(C,\emptyset)$, then the edges of $\Pc(A,\emptyset)\times \Pc(C,\emptyset)$ are $(k_i, 0)\R_{\geq 0}$ and $(0, k_j')\R_{\geq 0}$. 
\end{enumerate} 
\end{prop}

If no confusion shall arise, we may simply write the vectors $k_i$ to denote the edges of a pointed polyhedral cone. To prove that every element of a pointed polyhedral cone has a certain property, it is often enough to check the property for its edges. For example, if $\Pc(A,B)$ is a pointed polyhedral cone, and $C$ is a convex set, then $\Pc(A,B)\subseteq C$ if and only if the edges $k_i$ of $\Pc(A,B)$ lie in $C$.

\subsubsection{Minkowski's domain}
The next lemma describes $\mathcal{M}_n$ for $n\le 4$ as a pointed polyhedral cone. We refer to a proof that is quite elegant, but technical \cite[p. 257-258]{cassels2008rational}.

\begin{lemma}\label{nle4} For any $n\leq 4$, a quadratic form $q$ is a Minkowski reduced positive definite form if and only if the following hold:  
\begin{enumerate} 
\item $0<q_{11}\leq q_{22}\leq \cdots\leq q_{nn}$,
\item $q(x)\geq q_{kk}$ for $x=(x_1, \ldots, x_n) \in \R^n$ with coordinates $x_i \in \{-1, 0, 1\}$ for all $i=1, \ldots, n$ and satisfying $x_k=1$ and $x_j=0$ for $j>k$. 

\end{enumerate} 
\end{lemma}

Based on the preceding Lemma we give a precise description of $\cM_3$ as a polyhedral cone. 

\begin{theorem}[$\mathcal{M}_3$ as a polyhedral cone]\label{th: M3} The set of symmetric positive definite Minkowski reduced forms $q(x)=x^TQx$ in $3$-dimensions is precise those which satisfy 

\begin{equation} \label{eq:strict_m3} \begin{gathered} 0 <q_{11}, \\ 
0\leq q_{22}-q_{11}, \\
0\leq q_{33}-q_{22}, \\
0\leq q_{11}-2q_{12},\\
0\leq q_{11}+2q_{12}, \\
0\leq q_{11}+q_{22}+2q_{12}-2q_{13}-2q_{23}, \\
0\leq q_{11}+q_{22}-2q_{12}-2q_{13}+2q_{23}, \\ 
0\leq q_{11}+q_{22}-2q_{12}+2q_{13}-2q_{23}, \\ 
0\leq q_{11}+q_{22}+2q_{12}+2q_{13}+2q_{23},\\
0\leq q_{11}-2q_{13}, \\ 
0\leq q_{11}+2q_{13},\\
0\leq q_{22}-2q_{23}, \\ 
0\leq q_{22}+2q_{23}. 
\end{gathered} 
\end{equation} 
\end{theorem}

\begin{proof} This follows from Lemma \ref{nle4}. \end{proof}

The Minkowski domain is always a pointed polyhedral cone.  Unfortunately, in higher dimensions the number of inequalities \textit{explode}; see Tammela's list \cite[p. 20]{schurmann2009computational}, \cite{tammela1977minkowski}. 

\begin{theorem}[$\mathcal{M}_n$ is a polyhedral cone {\cite[p. 256-257]{cassels2008rational}}]\label{MinPcone} 
Only finitely many of the conditions for a quadratic form $q$ to be in $\mathcal{M}_n$ according to Proposition \ref{prop:posDmink} are non-redundant.  These conditions define $\mathcal{M}_n$ as a pointed polyhedral cone. 
\end{theorem}

Although the set of Minkowski reduced forms may contain more than one representative for each equivalence class, the representatives are \textit{almost} unique in the following sense. 

\begin{theorem}\label{th:alun} Assume that $q$ is in the interior of $\mathcal{M}_n$ and has associated matrix $Q$.  Then for any unimodular $B$, the quadratic form with associated matrix $B^T Q B$ is in $\mathcal{M}_n$ if and only if $B$ is diagonal.
\end{theorem}

\begin{proof}Write $B=[b_j]$, where $b_j$ are column vectors. Since the determinant of a unimodular matrix is $\pm 1$, the entries of the first column, $b_1$, satisfy $\gcd((b_1)_1,\ldots,(b_n)_1)$ $=1$.  Then, $q(b_1)=q(B e_1)>q_{11}$ unless $b_1=\pm e_1$. However, we know that $q(b_1)=q_{11}$, because by Minkowski reduction $q(B e_1)$ is a smallest non-zero value. This means that the lower right $(n-1)\times (n-1)$ matrix of $B$, say $B'$, has determinant $\pm 1$ and is also unimodular. For this reason, the second vector $b_2$ has $\gcd((b_2)_2,\ldots,(b_2)_n)$ $=1$. Therefore $q(b_2)>q_{22}$ unless $b_2=\pm e_2$.  The equality $q(b_2)=q_{22}$ follows from the second formulation of Minkowski reduction in terms of lattices; write $Q=A^TA, A=[a_j]$ and note that $AB=[\pm a_1\:Ab_2\:\cdots\:Ab_n]$. Since $a_1,Ab_2$ are part of a basis of $A\ZZ^n$, the length of $Ab_2$ must be equal to the length of $a_2$. We can repeat this process until each $b_i=\pm e_i$. 
\end{proof}

\begin{cor}\label{Mplus} The set $\mathcal{M}_n^+:=\mathcal{M}_n\cap \{q: q_{1j}\geq 0 \textnormal{ for each } 1\leq j \leq n\}$ contains a representative of each positive definite quadratic form. In the interior, the representatives are unique. 
\end{cor}

This corollary is a direct consequence of Theorem \ref{th:alun}.

\subsubsection{Successive minima} To understand one of the most striking results about Minkowski reduction, we require \em successive minima.  \em

\begin{defn} \label{def:succ_min} 
Assume that $Q$ is a positive definite symmetric $n \times n$ matrix.  The \em i:th successive minimum of Q \em is 
\begin{align*}
    \lambda_i(Q):= \min \left\{q \left(x^{(i)}\right):\; \right. & x^{(1)},\ldots,x^{(i)}\in\Z^n \textnormal{  are linearly independent, }  \\
 &    \left .\textnormal{ and }q \left(x^{(j)} \right)\le q\left(x^{(j+1)}\right)\textnormal{  for each }1\le j\le i-1 \right \}.
\end{align*} 
\end{defn} 

\begin{theorem}[van der Waerden, Satz 7 \cite{van1956reduktionstheorie}]\label{th:van} For a Minkowski reduced positive definite form $q$ we have
$$q_{ii}\le \Delta_i\lambda_i(Q)\quad \textnormal{\textit{with}} \quad \Delta_i:=\max\left\{1,\Big(\frac{5}{4}\Big)^{i-4}\right\}.$$
\end{theorem}

It is conjectured in \cite{schurmann2009computational} that there is an even tighter bound: $q_{ii}\le i\lambda_i(Q)/4$ for $i>5$. By Theorem \ref{th:van}, if two Minkowski reduced forms $q^{(1)},q^{(2)}$ are isospectral, then $q^{(1)}_{22}=q^{(2)}_{22}$. There is no guarantee however that $q^{(1)}_{33}=q^{(2)}_{33}$, as we observed in \S \ref{s:differentdonuts}.

\subsection{Schiemann's reduction}\label{sss:schiered} In his thesis \cite{schiemann1994ternare}, Schiemann introduced the following reduction, which he called the \textit{Vorzeichennormalform}.  We call it \textit{Schiemann reduction}. 

\begin{defn}[Schiemann reduction]\label{def: schie}
A ternary positive definite form $f$ is said to be \em Schiemann reduced \em if

$1a)$ $f$ is Minkowski reduced,

$1b)$ $f_{12}\geq 0$, and  $f_{13}\geq 0$,

$1c)$ $2f_{23}> -f_{22}$,\\
and the following facet conditions hold:

$2a)$ $f_{12}=0$ $\Longrightarrow f_{23}\geq 0$,

$2b)$ $ f_{13}=0$ $\Longrightarrow f_{23}\geq 0$,

$3a)$ $f_{11}=f_{22} \Longrightarrow |f_{23}|\leq f_{13}$,

$3b)$ $f_{22}=f_{33} \Longrightarrow f_{13}\leq f_{12}$,

$4a)$ $f_{11}+f_{22}-2f_{12}-2f_{13}+2f_{23}=0 \Longrightarrow f_{11}-2f_{13}-f_{12}\leq 0$,

$4b)$ $2f_{12}=f_{11}\Longrightarrow f_{13}\leq 2f_{23}$,

$4c)$ $2f_{13}=f_{11}\Longrightarrow f_{12}\leq 2f_{23}$,

$4d)$ $2f_{23}=f_{22}\Longrightarrow f_{12}\leq 2f_{13}$.
\\ 
\em Schiemann's domain, \em denoted $\choir$, is the set of Schiemann reduced forms embedded in $\RR^6$. Schiemann's pointed polyhedral cone, denoted $\choir_p$ is the set of elements of $\R^6$ that satisfy conditions 1a), 1b), and 1c). 
\end{defn} 

Schiemann's pointed polyhedral cone, $\choir_p$ is a non-closed pointed polyhedral cone that properly contains $\choir$.  
It has several facets. The facet conditions $2-4$ determine which forms of $\choir_p$ are contained in $\choir$.  Unfortunately, the facet conditions cannot be expressed according to the definition of polyhedral cone, which is why we cannot express $\choir$ itself as a pointed polyhedral cone. The more well-known set of Eisenstein reduced forms (see \cite[p. 142-143]{o1980indecomposable}) contains precisely one representative of each positive definite ternary quadratic form, and in fact, Schiemann used this theory to prove 

\begin{theorem} Any positive definite quadratic form in three variables has a unique representative that is Schiemann reduced.
\end{theorem}

The proof is not very interesting, it is long and technical \cite{schiemann1994ternare}. What is interesting is that the Schiemann reduction of a positive definite quadratic form is \em unique\em.

\begin{lemma}\label{ClosV} The closure of $\choir$ is a pointed polyhedral cone equal to the closure of $\mathcal{M}_3^+$.
Alternatively, it is given by the following system
\begin{equation} \label{SR_Form} 
\begin{cases}
0 \le q_{11} \le q_{22} \le q_{33} \\
0 \le 2q_{12} \le q_{11} \textnormal{ \& } 0 \le 2q_{13} \le q_{11} \\
-q_{22} \leq 2q_{23} \le q_{22} \\
0 \le q_{11} + q_{22} -2q_{12} - 2q_{13} + 2q_{23} 
\end{cases}
\end{equation}
\end{lemma}

One can calculate the edges of $\overline{\choir}$ for instance via a package like \texttt{Polymake}~\cite{gawrilow2000polymake} or by applying Section \ref{ss: CalcEdge}.  The set of edges is denoted by 
\[ M = M_1 \cup M_2 \cup M_3,\] 
\begin{align} \label{edges_m1m2}
M_1 & := \left\{ \left(\begin{smallmatrix} 0 & 0 & 0 \\ 0 & 0 & 0 \\ 0 & 0 & 1 \end{smallmatrix}\right) \right\},\quad M_2  := \left\{ \left(\begin{smallmatrix} 0 & 0 & 0 \\ 0 & 2 & \pm 1 \\ 0 & \pm 1 & 2 \end{smallmatrix}\right) \right\}\quad \& \\
M_3 & := \left\{  \left(\begin{smallmatrix} 2 & 0 & 0 \\ 0 & 2 & \pm 1 \\ 0 & \pm 1 & 2 \end{smallmatrix}\right), \left(\begin{smallmatrix} 2 & 0 & 1 \\ 0 & 2 & \pm 1 \\ 1 & \pm 1 & 2 \end{smallmatrix}\right), \left(\begin{smallmatrix} 2 & 1 & 0 \\ 1 & 2 & \pm 1 \\ 0 & \pm 1 & 2 \end{smallmatrix}\right), \left(\begin{smallmatrix} 2 & 1 & 1 \\ 1 & 2 & 0 \\ 1 & 0 & 2 \end{smallmatrix}\right), \left(\begin{smallmatrix} 2 & 1 & 1 \\ 1 & 2 & 1 \\ 1 & 1 & 2\end{smallmatrix}\right) \right\}. \label{edge_m3} 
\end{align}

Note that any element of $\overline{\choir}$ can be expressed as an $\RR_{\ge0}$-linear combination of elements in $M$. By definition of $\choir$, we can for three sets of vectors $\mathfrak{A}$, $\mathfrak{B}$, and $\mathfrak{C}$  write 
\[\choir=\Pc(\mathfrak{A},\mathfrak{B})\cap \{f\in\mathbb{R}^6:\forall (c,d)\in \mathfrak{C}:(c\cdot f=0\Rightarrow d\cdot f\geq 0)\}. \]
The sets $\mathfrak{A},\mathfrak{B}$ consist of vectors in $\R^6$ that correspond to the conditions $1a),b),c)$ in the definition of Schiemann reduced forms. The set of vectors in $\mathfrak{C} \subseteq \R^6 \times \R^6$ produce the facet conditions. The elements of these three sets can be written out explicitly using Theorem \ref{th: M3} and Definition \ref{def: schie}.  There are Schiemann reduced forms that have representatives contained in $\overline \choir \setminus \choir$.  The following algorithm provides a means to remove those superfluous representatives.  To state the algorithm, for $v\in \R^n$ we define the sets 
\begin{equation} \label{eq:v_geq0} v^{\geq 0}:=\{x\in \R^n: v\cdot x\geq 0\}, \quad v^{\bot}:=\{x\in \R^n: v\cdot x=0\}. \end{equation} 
The superfluous forms are removed by checking if the first or second factor lies in a facet of $\choir$ which has facet conditions, and then removing forms that do not satisfy the facet conditions. 

\begin{lemma}\label{le:viol} Assume that the polyhedral cone $T\subseteq \overline{\choir}\times \overline{\choir}$.  We define a sequence of polyhedral cones $T_i$ for $i\in \N_0$ by 
\begin{align*}
T_0 & :=\Big(\Pc(\mathfrak{A},\mathfrak{B})\times \Pc(\mathfrak{A},\mathfrak{B})\Big)\cap T,\\
T_i & :=T_{i-1}\cap \bigcap\limits_{(c,d)\in \mathfrak{C}: T_{i-1} \subseteq (c,0)^{\bot}} (d,0)^{\ge}\bigcap\limits_{(c,d)\in\mathfrak{C}:T_{i-1} \subseteq (0,c)^{\bot}} (0,d)^{\ge}.
\end{align*}
The sequence $T_i$ becomes stationary at some $i_0$.  We define $T_{\choir\times\choir}:=T_{i_0}$. Further, $\overline{T\cap(\choir\times\choir)}=\overline{T_{\choir\times\choir}}$ and $T\cap (\choir\times\choir)\subseteq T_{\choir\times\choir}\subseteq T$. 
\end{lemma}

\begin{proof} Clearly we have $T_{i+1}\subseteq T_{i}$ for each $i$. Let $C_1^i$, respectively $C_2^i$, be the set of $c$ belonging to a pair $(c,d) \in \mathfrak C$, such that $T_i\subseteq (c,0)^{\bot}$, respectively $T_i\subseteq (0,c)^{\bot}$. Since $T_i$ monotonically decreases, $C_1^i$ and $C_2^i$ monotonically increase. However $\mathfrak{C}$ is finite meaning that $C_1^i$ and $C_2^i$ converge and become stationary. It follows that $T_i$ must become stationary. 

We are left to show $T\cap (\choir\times\choir)\subseteq T_{\choir\times\choir}$. Take any $(f,g)\in T\cap (\choir\times\choir)$. It is clear that $(f,g)\in T_0$. Now say $(f,g)\in T_i$. When calculating $T_{i+1}$, note that $T_i\subseteq (c,0)^{\bot}$ only if $f\cdot c=0$ at which point we already know $f \cdot d \geq 0$ for $(c,d)\in \mathfrak{C}$, since $(f,g)\in \choir\times\choir$. We have the analogous situation for $T_i\subseteq (0,c)^{\bot}$. In either case, this implies $(f,g)\in T_{i+1}$. \end{proof}

Note that we always have $(T_{\choir\times\choir})_{\choir\times\choir}=T_{\choir\times\choir}$.


\subsection{The minimal sets}\label{Flute} 
To identify quadratic forms that have identical representation numbers, Schiemann defined minimal sets which induce a transitive relation. Our treatment here is a slight generalization of Schiemann's.

\begin{defn}[Minimal set] \label{def:minqx} Let $\mathcal{Q}$ be a finite set of symmetric quadratic forms in $n$-dimensions. First define the transitive relation
\[ x\preceq_\mathcal{Q} y \Leftrightarrow q(x)\leq q(y) \textnormal{ for each }q\in \mathcal{Q}. \] 
The set of minimal vectors with respect to a set $X\subseteq \Z^n$ is defined
\[ \MIN_\mathcal{Q}(X):=\{x\in X: y\not\preceq_\mathcal{Q} x \textnormal{ for all }y\in X\setminus\{x\}\}. \] 
In case $n=3$, and $\mathcal{Q}$ is the set $M$ of edges of $\overline{\choir}$, listed in \eqref{edges_m1m2} and \eqref{edge_m3}, we simply write $\MIN$ and $\preceq$.
\end{defn}

Note that $\MIN_{\mathcal{Q}}(X)$ consists of all $x\in X$ for which no other element in $X$ is smaller with respect to $\preceq_{\mathcal{Q}}$.  As an example, note that 
\[ \MIN (\ZZ_*^3)=\{(1,0,0)\}.\]

\begin{defn} For a set $\mathcal{Q}$ of $n$-dimensional quadratic forms, $q_1, \ldots, q_N$, recall its positive hull, 
\[\pos(\mathcal{Q}):=\left\{\sum_{i=1} ^N \lambda_iq_i: \lambda_i\ge 0, \quad q_i \in \mathcal{Q}\right\}.\] 
The set $ \mathcal{Q}$ is \em regular \em if (i) all its elements are rational and positive semi-definite, (ii) the positive hull of $\mathcal{Q}$ contains the identity matrix, and (iii) the embedding of $ \mathcal{Q}$ into $\RR^{n(n+1)/2}$ spans $\RR^{n(n+1)/2}$.
\end{defn}
One can check that the edges of $\overline{\choir}$ form a regular set. Next we characterize $\MIN_\mathcal{Q}$ and describe its most important properties.

\begin{lemma}\label{le: posH} We have $x\preceq_\mathcal{Q} y$ if and only if $q(x)\leq q(y)$ for each $q\in \pos(\mathcal{Q})$.  Assume that $\mathcal{Q}$ is a regular set of quadratic forms, and that $\MIN_{\mathcal{Q}}(X)$ $\neq$ $\emptyset$. The minimal set $\MIN_\mathcal{Q}(X)$ is the unique smallest set in $X$ such that for any $z\in X$ there is a $y\in \MIN_\mathcal{Q}(X)$ with $y\preceq_\mathcal{Q} z$.
\end{lemma}

\begin{proof} The first statement can be verified from the definitions.  To prove the second statement we proceed by contradiction.  Assume that for some $z\in X$, each $y\in \MIN_\mathcal{Q}(X)$ has $y\not\preceq_\mathcal{Q} z$. Then $z\not\in \MIN_\mathcal{Q}(X)$, meaning that there is an element $z^{(1)}\in X\setminus\{z\}$ with $z^{(1)}\preceq_\mathcal{Q}z$. If $z_1\in \MIN_\mathcal{Q}(X)$, then we arrive at a contradiction. Otherwise, we construct a sequence $z^{(i)}$ inductively by letting $z^{(i+1)}$ be an element in $X\setminus\{z,z^{(1)},\ldots,z^{(i)}\}$ such that $z^{(i+1)}\preceq_Q z^{(i)}$. By transitivity of $\preceq_\mathcal{Q}$, we would arrive at a contradiction if at any point $z^{(i)}\in \MIN_\mathcal{Q}(X)$. Since $I\in \pos(\mathcal{Q})$, by the first statement of the lemma we know that $\|z^{(i)}\|\leq \|z^{(i-1)}\|\leq \cdots \leq \|z\|$. Each $z^{(i)}\in \Z^n$ is distinct, therefore this sequence of norms must become stationary. This will eventually contradict $z^{(i+1)} \in X \setminus\{z,z^{(1)},\ldots,z^{(i)}\}$, because there are only finitely many vectors in $\ZZ^n$ of any given fixed length. To see that $\MIN_{\mathcal{Q}}(X)$ is the unique smallest set with this property, note that if $z\in \MIN_{\mathcal{Q}}(X)$, then the only element $y\in X$ with $y\preceq_{\mathcal{Q}} z$ is $z$ itself. 
\end{proof}

\begin{prop}\label{PropMIN} Let $\mathcal{Q}$ be a regular set. We have the following:

\begin{enumerate}[label=(\roman*)]
    \item $x\preceq_{\mathcal{Q}} y$ implies $\|x\|\leq \|y\|$.
    
    \item If $x\preceq_{\mathcal{Q}} y\preceq_{\mathcal{Q}} x$, then $x=\pm y$.
    
    \item $\MIN_{\mathcal{Q}}(X)$ is finite.
    
    \item If $0\not \in X\neq \emptyset $, and for any shortest vector $u$ of $X$ we have $-u\not\in X$, then $\MIN_{\mathcal{Q}}(X)\neq \emptyset$. 
    
    \item $\emptyset\neq X\subseteq \Z_*^n$ implies  $\MIN_{\mathcal{Q}}(X)\neq \emptyset$.
\end{enumerate}
\end{prop}

\begin{proof} $ $
    
(i) Follows from Lemma \ref{le: posH} and the fact that $I\in \pos\mathcal{Q}$.
    
(ii) We have $q(x)=q(y)$ for all $q\in \mathcal{Q}$. This can be written 
\[ 0=q(x)-q(y)=\sum_iq_{ii}(x_i^2-y_i^2)+\sum_{i<j}2q_{ij}(x_ix_j-y_iy_j))=0. \] 
    The fact that $\mathcal{Q}$ spans $\R^{n(n+1)/2}$ implies that $x_ix_j=y_iy_j$ for each $i,j$. Then $y_i=\pm x_i$ for each $i$.  In fact, it must be either $y_i = x_i$ for all $i$ or $y_i = - x_i$ for all $i$.  To see this, assume $x_1, x_2 \neq 0$, and $y_1 = x_1$, $y_2 = -x_2$.  This would contradict $x_1x_2=y_1y_2$.    
    
(iii) If $\MIN_{\mathcal{Q}}(X)$ were infinite, then there would be a sequence $x^{(i)}\in \MIN_{\mathcal{Q}}(X)$ of distinct elements such that $x^{(1)}\not\preceq_{\mathcal{Q}} x^{(2)}\not\preceq_{\mathcal{Q}} x^{(3)} \ldots$. We can find a subsequence $x^{(i_j)}$ such that there is a $q\in \mathcal{Q}$ with $q(x^{(i_1)})>q(x^{(i_2)})>q(x^{(i_3)})>\ldots$. The elements of $\mathcal{Q}$ are semi-positive and rational.  Consequently, for each $q \in \mathcal Q$, $q(\Z^n)$ is a discrete set.  So the sequence must become stationary, a contradiction. 
   
(iv) Denote the distinct vectors in $X$ that share the shortest length in $X$ by $u^{(1)},\ldots,u^{(k)}\in X$. Assume by contradiction that no $u^{(i)}$ is in $\MIN_{\mathcal{Q}}(X)$. Then for each $u^{(i)}$ there is an $x\in X\setminus \{u^{(i)}\}$ such that $x\preceq_{\mathcal{Q}} u^{(i)}$. By statement (i), $x=u^{(j)}$ for some $j\neq i$. If $k=1$, we directly have a contradiction. If $k>1$, then we could find two vectors $u^{(i)}\neq u^{(j)}$ such that $u^{(j)}\preceq u^{(i)}\preceq u^{(j)}$ as a consequence of the set of $u^{(i)}$ being finite. By (ii) this implies $u^{(j)}=\pm u^{(i)}$ which is a contradiction to the assumption in (iv).

(v) Follows directly from (iv) and the definition of $\Z^n_*$. 
\end{proof}

\subsubsection{Calculating $\MIN_{\mathcal{Q}}(X)$} It helps to know that any minimal set is finite given a regular set $\mathcal{Q}$, but how do we find its elements? The following lemma was used by Schiemann to reduce the calculation of $\MIN$ under certain conditions. 

\begin{lemma}\label{SchiMIN}
Let $\emptyset \ne Y \subseteq X \subseteq \Z_*^n$, and let $ \mathcal{Q}$ be regular. If 
\[ W \supseteq \{ x \in X : x \not\succeq_{\mathcal{Q}} y \:\:\forall y \in Y\} \cup Y,\] then $\MIN_{\mathcal{Q}}(X) = \MIN_{\mathcal{Q}}(X \cap W)$.
\end{lemma}

\begin{proof} We first show  $\MIN_{\mathcal{Q}}(X) \subseteq \MIN_{\mathcal{Q}}(X \cap W)$, and it's not hard to see that it suffices to prove $\MIN_{\mathcal{Q}}(X)\subseteq X\cap W$. Take $z\in \MIN_{\mathcal{Q}} (X)\subseteq X$. If $z\not\succeq_{\mathcal{Q}} y$ for all $y\in Y$, then by the assumption on  $W$, $z\in W$. If for some $y\in Y$ we have $z\succeq_{\mathcal{Q}} y$ then by construction of $\MIN_{\mathcal{Q}}(X)$, $y=z\in Y$, meaning that also in this case we have $z\in W$. 

Secondly we prove $\MIN_{\mathcal{Q}}(X) \supseteq \MIN_{\mathcal{Q}}(X \cap W)$. Let $z\in \MIN_{\mathcal{Q}}(X\cap W)$ and take any $x\in X$ such that $x\preceq_{\mathcal{Q}} z$. We show that $x$ is necessarily equal to $z$, since then $z\in \MIN_{\mathcal{Q}}(X)$ by definition. If $x\in X\cap W$, then since $z\in \MIN_{\mathcal{Q}}(X\cap W)$, we have $z=x$. If $x\in X\setminus W$, then it follows that there is a $y\in Y$ such that $y\preceq_{\mathcal{Q}} x$. By transitivity, $y\preceq_{\mathcal{Q}} z$. Now since $Y\subseteq X\cap W$, this implies $y=z$. As a consequence, $z\preceq_{\mathcal{Q}} x\preceq_{\mathcal{Q}} z$ which by Proposition \ref{PropMIN} means $z=x$.\end{proof}

To determine $\MIN_{\mathcal{Q}}(X)$ we want to find appropriate, finite $Y$ and $W$ as above. We present a new result for this purpose. First, let $\pi:\RR^n\to \RR^{n-1}$ denote the projection onto the last $n-1$ coordinates.

\begin{theorem}\label{CalcMINq} Assume that $\mathcal{E}$ is a regular set of $n$-dimensional quadratic forms.  Let $\mathcal{E}'$ be the set of non-invertible matrices of $\mathcal{E}$, with first row and column deleted.  Let $\mathcal{E}^{\times}$ be the invertible matrices of $\mathcal{E}$.  If $\mathcal{E}'$ is regular, then for $a\in \mathbb{Z}$ and $X\subseteq \mathbb{Z}^n$ such that $\MIN_{\mathcal{E}'}(\pi(X))\neq \emptyset$, we define $Y(a):= \{(a,z):z\in \MIN_{\mathcal{E}'}\big(\pi(X)\big)\}$.  Then,
\begin{align*}
    W(a):&=\big\{x\in X: t\not\preceq_{\mathcal{E}} x \textnormal{ for each }t\in Y(a)\big\}\cup Y(a)\\
    &\subseteq \big\{x\in \mathbb{Z}^n:\lambda_{\min}\|x\|^2<\lambda_{\max}\max_{t\in Y(a)}\{\|t\|^2\}\big\} \cup Y(a),
\end{align*}
where $\lambda_{\min}$, $\lambda_{\max}$ are respectively the smallest and largest eigenvalues among elements of $\mathcal{E}^\times$. In particular, $W(a)$ is a finite set.
\end{theorem}
\begin{proof} Take any $x\in W(a) \setminus Y(a)$ and note that if $f\in \mathcal{E}'$, then $f(x)=f(\pi(x))$. By assumption and Proposition \ref{le: posH}, there is an element $z$ of $\MIN_{\mathcal{E}'}(\pi(X))$ such that $z\preceq_{\mathcal{E}'}\pi(x)$. Therefore, if $(a,z)\not\preceq_\mathcal{E}x$ we would require $g(a,z)>g(x)$ for some $g\in \mathcal{E}^\times$. It follows that $\lambda_{\min}\|x\|^2<g(a,z)\leq \max \{g(t)\}\le \lambda_{\max}\max\{\|t\|^2\}$, where the maximum is taken over $t\in Y(a)$. This proves the inclusion, since $\lambda_{\min}>0$, $\lambda_{\max}/\lambda_{\min}\ge 1$, and we end by noting that both $\mathcal{E}^\times$ and $Y(a)$ are finite.\end{proof}

If $\mathcal{E}=M$ is the set of edges of $\overline{\choir}$, then the conditions of the theorem are satisfied. 
To apply the theorem in Schiemann's symphony, let
\begin{equation} \label{eq:x0x1x2} \begin{gathered}  X_0:=\ZZ_\ast^3\setminus e_1\ZZ, \quad X_1:=\ZZ_\ast^3\setminus (e_1\ZZ+e_2\ZZ), \\  X_2:=\ZZ_\ast^3\setminus \big((e_1\ZZ+e_2\ZZ)\cup  (e_1\ZZ+e_3\ZZ)\big). \end{gathered} \end{equation} 
In Theorem \ref{CalcMINq}, the corresponding 
\begin{align}
&Y_0(a)=\{ (a,1,0) \} , \quad Y_1(a)=\{(a,0,1)\}, \quad Y_2(a)=\{(a,1,1), (a, -1, 1)\},  \label{eq:MinkApp1}\\
   & W_i(a) \subseteq \left\{x : \|x\| <2 \sqrt{2(a^2+2)}\right\}.
   \label{eq:MinkApp2}
\end{align}
Schiemann showed similar, sharper bounds through less general means. We could find sharper bounds by determining $\max\{g(t)\}$ as in the proof above, instead of estimating using $\lambda_{\max}$. In Schiemann's symphony, we calculate the minimal sets of $X_i\setminus Z$ for finite sets $Z$. Since $Z$ is finite, we can always find an appropriate $a$ such that $Y_i(a)\subseteq X\setminus Z$ and in this way apply Lemma \ref{SchiMIN}.

Although we work in 3 dimensions here, the statements can be generalized to any dimension without problems, essentially by replacing $\overline{\choir}$ by the closure of $\mathcal{M}_n^+$.  
\begin{defn}\label{Klav}
Let $x^{(1)},\ldots, x^{(k)} \in X\subseteq \ZZ_\ast^3$ be distinct elements. We define 
\begin{equation*} \begin{gathered} \csec \left(X, \{x^{(i)}\}_{i=1}^k\right) := \\ 
\left\{f \in \overline{\choir} : f\left(x^{(j)}\right) = \min \left \{ f \left(X \setminus \left\{x^{(1)}, \ldots, x^{(j-1)}\right\}\right) \right\} \:\:\forall j = 1 ,\ldots ,k\right\}. \end{gathered} \end{equation*} 
We set $\csec(X,\emptyset) := \overline{\choir}$.  
\end{defn}

Above, $\csec \left(X, \{x^{(i)}\}_{i=1}^k\right)$ is the set of quadratic forms in $\overline{\choir}$ for which $x^{(1)},\ldots,x^{(k)}$ yield the successively smallest values of the forms evaluated at elements of $X$.  To show the connection to the representation numbers, let $\{x^{(i)},x\}_{i=1}^k$ denote the sequence $x^{(1)},\ldots,x^{(k)},x$ in this order.  If 
$f\in  \csec(X, \{x^{(i)},x\}_{i=1}^k)$, then
\begin{align}
    f(x)&=\min f \left( X\setminus \left\{x^{(1)},\ldots,x^{(k)}\right\} \right)\label{eq: klav1}\\
    &=\min \left\{t_0\in \R_{\ge 0}: \sum_{0\leq t\leq t_0}\mathcal{R}_{X}(f,t)\geq k+1\right\}\label{eq: klav2}. 
\end{align}
The sets $\csec \left(X, \{x^{(i)}\}_{i=1}^k\right)$ are pointed polyhedral cones contained in the closure of Schiemann's domain.

\begin{lemma} Let $X\subseteq \Z_*^3$ and $x^{(1)},\ldots,x^{(k)}\in X$ be distinct. Assume that $\MIN(X \setminus \{x^{(i)}\}_{i=1}^k)\neq \emptyset$. Then $\csec(X, \{x^{(i)}\}_{i=1}^k)$ is a pointed polyhedral cone, specifically 
\begin{align*}
    & \csec(X, \{x^{(i)}\}_{i=1}^k) = \\
&    \left\{ f \in \overline{\choir} : f(x^{(1)}) \le \cdots \le f(x^{(k)})  \textnormal{ and } 
    f(x^{(k)}) \le f(x) \; \forall x \in \MIN(X \setminus  \{x^{(i)}\}_{i=1}^k)\right\}.
\end{align*}
\end{lemma}

To see how the sets $\csec \left(X, \{x^{(i)}\}_{i=1}^k\right)$ are realized as polyhedral cones in a computer, we proceed as follows. For fixed $x,y\in \R^n$ and any quadratic form $f$ we have 
\[ f(y) \geq f(x) \iff \sum_{i}f_{ii}(y_i^2-x_i^2) + \sum_{i < j}f_{ij}(2y_iy_j-2x_ix_j) \ge 0. \]  
The right hand side is a linear inequality for the values of $f_{ij}$.  Then, given $x^{(i)}$ and $\MIN(X \setminus \{x^{(i)}\}_{i=1}^k)$, which we calculated in the last section, we note that $\MIN(X \setminus \{x^{(i)}\}_{i=1}^k)$ is a finite set, and it is therefore possible to explicitly write out all the linear inequalities defining the set $\csec \left(X, \{x^{(i)}\}_{i=1}^k\right)$. 

\begin{lemma}\label{KlavUni}
Let $X\subseteq \Z_*^3$ and $x^{(1)}, \ldots, x^{(k)}\in X$ be distinct. If $X\setminus\{x^{(i)}\}_{i=1}^k\neq \emptyset$, then
\[  \csec(X,\{ x^{(i)}\}_{i=1}^k) =\bigcup_{y \in \MIN(X \setminus \{x^{(i)}\}_{i=1}^k)} \csec(X, \{x^{(i)},y\}_{i=1}^k).\] 
\end{lemma}

\begin{proof} $ $

``$\supseteq$'': This inclusion is immediate since $\csec(X,\{ x^{(i)},y\}_{i=1}^k)\subseteq \csec(X,\{x^{(i)}\}_{i=1}^k)$ by definition.

``$\subseteq$'': Let $f\in \csec(X, \{x^{(i)}\}_{i=1}^k)$ and define
\[ Y_f:=\left\{y\in X\setminus\{x^{(i)}\}_{i=1}^k: f(y)=\min f(X\setminus\{x^{(i)}\}_{i=1}^k)\right\}. \] 
The set $Y_f$ is non-empty by the discreteness of $f(X\setminus\{x^{(i)}\}_{i=1}^k)$. 
By Proposition \ref{PropMIN} (v), $\MIN(Y_f)\neq \emptyset$. Fix some $y\in \MIN(Y_f)$. We proceed to show that $y\in \MIN(X\setminus\{x^{(i)}\}_{i=1}^k)$. In that case, we are done since by definition we have $f\in \csec(X, \{x^{(i)},y\}_{i=1}^k)$. We consider the following decomposition,
\[X\setminus\{x^{(i)}\}_{i=1}^k=Y_f\cup\left(\left(X\setminus\{x^{(i)}\}_{i=1}^k\right)\setminus Y_f\right). \]
To see that $y\in \MIN(X\setminus\{x^{(i)}\}_{i=1}^k)$, it suffices to show that $x\not\preceq y$ for any $x\in X\setminus\{x^{(i)},y\}_{i=1}^k$. Since $y\in \MIN(Y_f)$, we know that $x\not\preceq y$ for $ x\in Y_f\setminus\{y\}$. Now consider $x\in \left(X\setminus\{x^{(i)}\}_{i=1}^k\right)\setminus Y_f$, and note that $x\neq y$.  If $\left(X\setminus\{x^{(i)}\}_{i=1}^k\right)\setminus Y_f=\emptyset$, then we are done. Otherwise, for such an $x$,  $f(x)>\min f(X\setminus\{x^{(i)}\}_{i=1}^k) $, since $x\not\in Y_f$. This means $f(x)>f(y)$, because $y\in Y_f$ which implies $x\not \preceq y$, completing the proof. \end{proof}

\subsection{Schiemann's symphony} \label{Expl} We are almost ready to perform Schiemann's symphony, the algorithm that proves $\flat_3=1$; three-dimensional flat tori are determined by their spectra. For this we define the \em duets, \em that are pairs of isospectral Schiemann reduced forms 
\[\duets:=\{(f,g)\in \choir\times\choir:f \textrm{ and  $g$ are isospectral} \}.\] 
The goal of Schiemann's symphony is to prove that all duets are in fact \em solos, \em meaning the two forms are identical.  We therefore aim to prove that $\duets \subset \solos$, where
\[ \solos := \{(f,f):f\in\RR^{6}\}.\]

\begin{defn}[Coverings and refinements] \label{def:cover} 
For $A \subseteq \R^n$, $P$ is a \em covering \em of $A$ if 
\[ A \subseteq \bigcup_{U \in P} U.\] 
If $P'$ is also a covering of $A$, it is a \em refinement \em if each element of  $P'$ is contained in some $U \in P$. 
\end{defn} 

To prove that $\duets \subset \solos$, from which it immediately follows that $\flat_3=1$, we use a computer algorithm to calculate polyhedral cones that cover $\duets$ and iteratively refine these coverings to prove that eventually the coverings are all contained in $\solos$.

\begin{defn}[In tune] \label{acc} A polyhedral cone $T$ is \textit{in tune with respect to $\Lambda(T), k=k(T)$ and sequences $x^{(1)},\ldots,x^{(k)},y^{(1)},\ldots,y^{(k)} $} if the following properties hold

\textbf{P1:} $T\subseteq \overline{\choir}\times \overline{\choir}$ is a polyhedral cone and $T_{\choir\times\choir}=T$ as in Lemma \ref{le:viol}.

\textbf{P2:} $\Lambda:=\Lambda(T)$ is the largest of the three nested subsets of $\R^3$,
\[ e_1\ZZ \subseteq e_1\ZZ+e_2\ZZ \subseteq (e_1\ZZ+e_2\ZZ)\cup (e_1\ZZ+e_3\ZZ)\]
such that $f|_{\Lambda(T)}=g|_{\Lambda(T)} $ for all $(f,g)\in T$. 

\textbf{P3:} We have $x^{(i)},y^{(i)}\in \ZZ_*^3\setminus\Lam$ and 
\begin{align*}
T & \subseteq \csec(\Z_*^n\setminus \Lambda,\{x^{(i)}\}_{i=1}^k)\times \csec(\Z_*^n\setminus \Lambda,\{y^{(i)}\}_{i=1}^k\}),\\  
T & \subseteq \{(f,g)\in \overline{\choir}\times \overline{\choir}: f(x^{(i)})=g(y^{(i)})\:\: \forall i=1,\ldots,k\}.
\end{align*}
A covering $\mathcal{T}$ of $\duets$ is \textit{in tune} if each $T\in \mathcal{T}$ is in tune.
\end{defn}

Aiming to define a sequence of coverings of $\duets$, we start with 
\begin{equation}\label{eq: firstcover}
   \mathcal{T}_0:=\left\{\{(f,g)\in \choir_p \times \choir_p   :f_{11}=g_{11}\}_{\choir\times\choir}\right\}.  
\end{equation}
Denote by $T$ the single element of $\mathcal{T}_0$. Observe that $\duets\subseteq T$, by Minkowski reduction and Corollary \ref{cor:dim_vol}. The set $T$ satisfies P1, P2 and P3 by setting $k=1$, $x^{(1)}=y^{(1)}=e_1$, and $\Lambda=e_1\ZZ$. For an in tune covering $\mathcal{T}$ of $\duets$, we define a refinement as follows.

\begin{defn} \label{def:refine} Let $T$ be an in tune polyhedral cone with corresponding $\Lambda,k$,$x^{(i)}$, and $y^{(i)}$ as in Definition \ref{acc}. We define its \em refinement \em $\audit_T$ according to the following to cases.  

\textbf{Case 1.} $T\subseteq \solos$: Let $\audit_T:=\{T\}$. The set $T\in\audit_T$ is in tune with $\Lambda,k,x^{(i)},y^{(i)}$.  

\textbf{Case 2.} $T\not\subseteq \solos$: Write $\twonotes(\Lam,\{z^{(i)}\}):=\MIN((\Z_*^n\setminus \Lambda)\setminus\{z^{(i)}\}_{i=1}^k)$ for a sequence $z^{(i)}$. Let for each $x\in \twonotes(\Lam,\{x^{(i)}\})$ and $y\in \twonotes(\Lambda,\{y^{(i)}\})$,
\begin{align*}
    S_{xy} & := T\cap \big(\csec(\Z_*^n\setminus \Lambda,\{x^{(i)},x\}_{i=1}^k)\times \csec(\Z_*^n\setminus \Lambda,\{y^{(i)},y\}_{i=1}^k)\big),\\
    T_{xy} & := [S_{xy}\cap \{(f,g)\in \overline{\choir}\times \overline{\choir}: f(x)=g(y)\}]_{\choir\times\choir}.
\end{align*}
We define
$$\audit_T:=\bigcup_{\substack{x\in \twonotes(\Lam,\{x^{(i)}\}), \;y\in \smaller{\twonotes}(\Lambda,\{y^{(i)}\})}}\{T_{xy}\}.$$
Each $T_{xy}$ is in tune with variables as follows.  Let $\Lambda_{xy}=\Lambda(T_{xy})$ be maximal with $f|_{\Lambda_{xy}}=g|_{\Lambda_{xy}}$ for all $(f,g)\in T_{xy}$. Let $k(T_{xy})=k(T)+1$ and $x^{(k+1)}=x,y^{(k+1)}=y$. If $\Lambda_{xy}\neq \Lambda$, then let $0\le r\le k(T)+1$ be maximal with $\#(\{x^{(i)}\}_{i=1}^r\setminus \Lambda_{xy})=\#(\{y^{(i)}\}_{i=1}^r\setminus \Lambda_{xy})$, and set instead $k(T_{xy})=r$. In this case, $T_{xy}$ is in tune with $x^{(i)},y^{(i)}$ for $i=1,\ldots,r$. 
\end{defn}

Observe that $T_{xy}\subseteq S_{xy}$, and $(f,g)\in T_{xy}$ implies $(f,g)\in T$, and $f(x)=g(y)$. The sets $S_{xy}$ decompose $T$ into possibly overlapping subsets whose union is equal to $T$. Passing from $S_{xy}$ to $T_{xy}$, multiple representatives of the same equivalence class are removed according to Lemma \ref{le:viol}. We are now ready to define the sequence $\mathcal{T}_i$ of coverings of $\duets$.

\begin{defn} \label{def:refine2} If $\mathcal{T}$ is a covering of $\duets$, then we define its refinement as 
\[ \mathcal{T}':=\bigcup_{T\in \mathcal{T}}\audit_T. \] 
With $\mathcal{T}_0$ as in Equation \ref{eq: firstcover}, we define the sequence $\mathcal{T}_i$ by $\mathcal{T}_{i+1}=\mathcal{T}_i'$ for each $i\geq 0$.
\end{defn} 
We must now argue that $\mathcal{T}'$ is actually a refinement of a given covering $\mathcal{T}$. This is done via the following proposition.

\begin{prop} $\mathcal{T}_i$ is a sequence of coverings of $\duets$, and each iteration is a refinement of the previous one.
\end{prop}

\begin{proof} Assume that $\mathcal{T}=\mathcal{T}_i$ for some $i$ is a covering of $\duets$. We show that $\mathcal{T}'=\mathcal{T}_{i+1}$ is a refinement. Fix some arbitrary $T\in \mathcal{T}$. Let us first note that by Lemma \ref{KlavUni},
\begin{align*}
    \bigcup_{\substack{x\in \twonotes(\Lam,\{x^{(i)}\}),\; y\in \twonotes(\Lam,\{y^{(i)}\})}}&\csec(\Z_*^3\setminus \Lambda,\{x^{(i)},x\}_{i=1}^k)\times \csec(\Z_*^3\setminus \Lambda,\{y^{(i)},y\}_{i=1}^k)
\end{align*}
is equal to $\csec(\Z_*^3\setminus \Lambda,\{x^{(i)}\}_{i=1}^k)\times \csec(\Z_*^3\setminus \Lambda,\{y^{(i)}\}_{i=1}^k)$. By property P3 of $T$, we have
\begin{align*}
    \bigcup_{x\in \twonotes(\Lam,\{x^{(i)}\}),\; y\in \twonotes(\Lam,\{y^{(i)}\})}S_{xy} &= T\cap  \csec(\Z_*^3\setminus \Lambda,\{x^{(i)}\}_{i=1}^k)\times \csec(\Z_*^3\setminus \Lambda,\{y^{(i)}\}_{i=1}^k)\\
    &=T. 
\end{align*}
Consider $(f,g)\in S_{xy}\cap \duets$. We have $\mathcal{R}(f,t)=\mathcal{R}(g,t)$ for all $t\in \R_{\geq 0}$. Since $f|_{\Lambda}=g|_{\Lambda}$, $\mathcal{R}_{\Z_*^3\setminus \Lambda}(f,t)=\mathcal{R}_{Z_*^3\setminus \Lambda}(g,t)$ for all $t\in \R_0^+$. By Equation \ref{eq: klav2} and the definition of $S_{xy}$, we have 
\begin{align*}
    f(x)& =\min \Big\{t_0\in \R_0^+: \sum_{0\leq t\leq t_0}\mathcal{R}_{\Z_*^3\setminus \Lambda}(f,t)\geq k+1\Big\}= \\
    &=\min \Big\{t_0\in \R_0^+: \sum_{0\leq t\leq t_0}\mathcal{R}_{\Z_*^3\setminus \Lambda}(g,t)\geq k+1\Big\}=g(y).
\end{align*}
This implies $(f,g)\in T_{xy}$ by Lemma \ref{le:viol}, and $S_{xy}\cap  \duets\subseteq T_{xy}$. Consequently, 
\[
T\supseteq \bigcup_{x\in \twonotes(\Lam,\{x^{(i)}\}),\; y\in \twonotes(\Lam,\{y^{(i)}\})}T_{xy} \supseteq \bigcup_{x\in \twonotes(\Lam,\{x^{(i)}\}),\; y\in \twonotes(\Lam,\{y^{(i)}\})}\big(S_{xy}\cap  \duets\big)= T\cap  \duets.
\]
By the fact that $\mathcal{T}$ is a covering of $\duets$, we have 
\[
\bigcup_{T\in \mathcal{T}}T\supseteq\bigcup_{T\in \mathcal{T}}\Big(\bigcup_{T_{xy}\in\audit_T}T_{xy}\Big)\supseteq \bigcup_{T\in \mathcal{T}}\big(T\cap  \duets\big) = \duets\cap \bigcup_{T\in \mathcal{T}}T=  \duets. \qedhere \]
\end{proof}

If $\mathcal{T}_i$ becomes stationary, by Definitions \ref{def:refine} and \ref{def:refine2} each $T\in \mathcal{T}_i$ lies in $\solos$, and we have $ \duets\subseteq \solos$. We cannot a priori conclude that this will occur, but we can show that the pairs of the coverings become more and more isospectral in the following sense. 

\begin{lemma} We have 
\begin{equation}\label{eqMet}
    \bigcap_{i\in \N}\bigcup_{T\in \mathcal{T}_i}T\cap (\choir\times\choir) = \duets.
\end{equation}
\end{lemma}
\begin{proof}[Proof sketch] By definition, $\duets$ is contained in the left side of \eqref{eqMet}.  So, assume that $(f,g)$ is in the left side of \eqref{eqMet}, we wish to show that $(f,g)\in \mathscr{D}$.  
There is a sequence of $T_i\in \mathcal{T}_i$ such that $(f,g)\in T_i$ and $T_i\subseteq T_{i+1}$ for each $i$ by construction. Since $\Lambda(T_i)$ increases monotonically and takes three values, it becomes stationary and equals, say $\Lambda$. Clearly, for $x\in \Lambda$, we have $f(x)=g(x)$. Next we construct a bijection $\phi$ from $\ZZ_\ast^3\setminus\Lambda$ to itself, such that $f(x)=g(\phi(x))$ for each $x\in \ZZ_\ast^3\setminus\Lambda$. The existence of $\phi$ proves that $f$ and $g$ are isospectral. There are sequences $x^{(i)},y^{(i)}\in \ZZ_\ast^3\setminus\Lambda$ for which $T_i$ are in tune with, meaning $f(x^{(k)})=\min f((\ZZ_\ast^3\setminus\Lambda)\setminus \{x^{(i)}\}_{i=1}^{k-1})$ for each $k$, and the analogous for $g(y^{(k)})$. Note that $\ZZ_\ast^3\setminus\Lambda=\{x^{(i)}\}_{i=1}^{\infty}=\{y^{(i)}\}_{i=1}^{\infty}$, or we'd arrive at a contradiction by Proposition \ref{PosQ} (5). In addition, $f(x^{(k)})=g(y^{(k)})$ for each $k$, and so defining $\phi(x^{(k)})=y^{(k)}$ finalizes the proof. 
\end{proof}


\subsubsection{Calculating edges}\label{ss: CalcEdge} Schiemann's symphony is performed by calculating $\mathcal{T}_i$ in Definition \ref{def:refine2} using a computer.  An algorithm calculates the edges of polyhedral cones in order to find $\Lambda$ in Definition \ref{acc} and to check the termination criterion $T\subseteq \mathscr{S}$ of Definition \ref{def:refine}. The statements here follow from the classic literature of \cite[chapters 2 and 3]{grunbaum1967convex}. For the complete proofs we refer to Schiemann's thesis \cite{schiemann1994ternare}; here we aim to give an understanding of the general mechanism. Consider linearly independent vectors $a_1,\ldots,a_n\in \RR^n$. To find the edges of $\Pc(\{a_i\}_{i=1}^n, \emptyset)$ we calculate all the kernels of the $n$ different $(n-1)\times n$ matrices we obtain from the set of $a_i$ by removing one those vectors. This approach is very computationally expensive for a larger set of constraints $a_i$, and therefore we proceed as follows. Assume that we know the edges $k_1,\ldots,k_r$ of some polyhedral cone $\Pc(A,B)$. What can we say about the edges of a polyhedral cone $\Pc(A\cup\{v\},B)\subseteq \R^n$ for some vector $v$?  Since $\Pc(A\cup B,B)=\Pc(A,B)$, we can without loss of generality assume that $B\subseteq A$. In this case, $\overline{\Pc(A,B)}=\Pc(A,\emptyset)$.

\begin{theorem}\label{th: CalcEdg} Let $\Pc(A,B)\neq \emptyset $ be a pointed polyhedral cone such that $B\subseteq A$. Let $k_1,\ldots, k_r$ be the edges of $\Pc(A,B)$. For a non-zero vector $v$ and the set $\Pc(A\cup\{v\},B)$, we have
\begin{enumerate}
    \item[Case 1:] If $k_i\in v^{\geq 0}$ for each $i$, then the edges of $\Pc(A\cup\{v\},B)$ are $k_1,\ldots, k_r$.
    
    \item[Case 2:] If $k_i\not \in v^{\geq 0}$ for some $i$ and each $k_j$ has $k_j\cdot v\leq 0$, then proceed as follows. Let $k_1',\ldots,k_l'$ be those edges among $k_1,\ldots, k_r$ that lie in $v^{\bot}$. If there are no such $k_i'$, then $\Pc(A\cup\{v\},B)$ is either empty or equal to $\{0\}$. The set $\Pc(A\cup\{v\},B)$ is empty if and only if $k:=\sum k_i'$ has $k\cdot b=0$ for some $b\in B$. If it is non-empty and non-zero, then its edges are $k_1',\ldots,k_l'$.
    
    \item[Case 3:] If $k_i\not \in v^{\geq 0}$ for some $i$ and some $k_j$ has $k_j\cdot v>0$, then proceed as follows. The set $\Pc(A\cup\{v\},B)$ is non-empty and its edges are calculated as those edges among $k_1,\ldots,k_r$ such that $k_i\cdot v\geq 0$ and the elements of the set 
\[ \left\{F\cap v^{\bot}: F=k_1\R_{\ge 0}+k_2\R_{\ge 0} \textnormal{ is a }2\textnormal{-face of }  \Pc(A,\emptyset)\textnormal{ with }k_1\cdot v>0, k_2\cdot v<0\right\}. \]
\end{enumerate}
\end{theorem} 

In our algorithm, we calculate a great number of edges.  Theorem \ref{th: CalcEdg} reduces the computing time since it allows us to do it cumulatively, by keeping track of all the edges at all times. Next we give a computable criterion for finding the 2-faces for Case 3 above.

\begin{lemma} Let $k_1\neq k_2$ represent different edges of a polyhedral cone $\Pc(A,\emptyset)$. Let $\{a_1,\ldots, a_r\}= \{a\in A:k_1\subseteq a^{\perp}\}$ and $\{a_1',\ldots,a_s'\}=\{a\in A : k_2\subseteq a^{\perp}\}$. Then $k_i\R_0^++k_j\R_0^+$ is a $2$-face of $\Pc_c(A,\emptyset)$ if and only if 
\[ \dim \bigcap_{a \in \{a_1,\ldots , a_r\}\cap\{a_1'\ldots ,a_s'\}}a^{\perp}=2. \] 
\end{lemma}

At each step, we want the number of elements of $A$ (and $B$) to be as few as possible to speed up the algorithm. The next lemma gives a simple condition with which we can remove some of the redundant constraints. 

\begin{lemma} Let $\Pc(A,\emptyset)$ be a pointed polyhedral cone of dimension $d$ and with edges $k_1,\ldots,k_r$. We have $\Pc(A,\emptyset)=\Pc(A',\emptyset)$ for
$$A':=\{c\in A: \#\{k_i:k_i\in c^{\bot}\}\geq d-1\}. $$
\end{lemma}

\begin{center}
\begin{table} State after the $i$:th iteration \\ $ $\\ 
\begin{tabular}{ cc }   
\begin{tabular}{ |c|c|c| }  
\hline
$i$ & HH:MM & \# Cones \\ \hline
 1& 00:00  & 1 \\ \hline
   2& 00:00  & 4 \\ \hline
 3& 00:00 & 42 \\ \hline
   4& 00:00  & 500 \\ \hline
 5& 00:02 & 3,311 \\ \hline
   6& 00:05  & 11,164 \\ \hline
  7& 00:28  & 31,334 \\ \hline
\end{tabular} &  
\begin{tabular}{ |c|c|c| } 
\hline
$i$ & HH:MM & \# Cones \\ \hline
  
 8& 00:59 & 59,970 \\ \hline
    9& 01:48  & 34,658 \\ \hline
 10& 02:22 & 4,452 \\ \hline
   11& 02:42  & 1,284 \\ \hline
 12& 02:53 & 702 \\ \hline
   13& 03:00  & 18 \\ \hline
 14& 03:01 & 0 \\ \hline
\end{tabular} \\
\end{tabular}
\caption{This table shows our computational results for Schiemann's symphony. Here, HH:MM correspond to the time in hours and minutes after the $i$:th iteration, and the number of cones is to the number of elements of $\mathcal{T}_i$. 
}  \label{table1}
\end{table}
\end{center}

\subsubsection{Results from the algorithm}  The finale of Schiemann's  symphony is the following result.

\begin{theorem} The sequence $\mathcal{T}_i$ becomes stationary for $i\geq 14$. Further, we have that each set $T\in \mathcal{T}_{14}$ lies in $\solos$.

\end{theorem}

As we have previously noted, this completes the symphony.  In other words, we have shown $\flat_3=1$. As documented in Table \ref{table1}, with one processor it took about $3$ hours to finish, and at least 147,442 polyhedral cones were computed. With 50 processors the algorithm took 19 minutes. We wrote the code in \texttt{Julia} \cite{bezanson2012julia} with the following packages and specifications of our computer:  \textbf{CPU:} Intel(R) Xeon(R) Platinum 8180 CPU @ 2.50 Ghz; \textbf{OS:} Fedora 32; \textbf{packages:} Abstract Algebra v.0.9.0, Nemo v.0.17.0 \& Hecke v.0.8.0. We found that the calculation of the minimal sets was not very time consuming in comparison to calculating edges. One way to make this program faster would be to minimize the amount of edges that need to be calculated. Here, we perform the symphony at a comfortable andante pace.  In his thesis, Schiemann performed the symphony at a quicker vivaci tempo by incorporating clever tricks to speed up the algorithm which we have not included here so as to keep the focus on the main arguments of the proof.




\section{Open problems and food for thought} \label{s:conjectures} 
The following question is, to the best of our knowledge, open. 

\begin{question} What are the precise values of the choir numbers $\flat_n$ for $n \geq 4$?  
\end{question} 

With the three equivalent but different perspectives in mind, this question is related to many open problems in number theory concerning the classification of quadratic forms as well as open problems concerning the geometry of lattices.  We take this opportunity to highlight some of these interrelated problems. 

\subsection{Asymptotics of the choir numbers} \label{ss: asymp} The celebrated geometer Wolpert, one of the three authors who proved \em one cannot hear the shape of a drum \em \cite{gww}, studied the moduli space of flat tori in \cite{wolpert}.  This space consists of the set of equivalence classes of flat tori, where all members of the same equivalence class are isometric. He gave a geometric description of this moduli space and proved that it is in a certain sense an unusual phenomenon that flat tori of different shapes are isospectral; the isometry class of a \textit{generic}  flat torus is determined by its spectrum. In his article, he showed that if each flat torus in a \textit{continuous family}, meaning a one-parameter family of flat tori that has a continuous family of basis matrices, is isospectral to any other  in the family, then all are isometric. We obtain Wolpert's result as a consequence of the following

\begin{lemma}\label{le: seq} Consider for $k\in \N$ the sequence of full-rank lattices $\G_k=A_k\Z^n$. Assume that all $\T_{\G_k}$ are mutually isospectral, and that $A_k\rightarrow A$.
At some point the sequence $\G_k$ becomes stationary up to congruency.
\end{lemma}

\begin{proof} The positive definite quadratic forms $A_k^TA_k$ all have the same image over $\Z^n$. By Corollary \ref{cor:dim_vol} and continuity of the determinant, $\det(A)=\det(A_k)$ for each $k$. Therefore, $\G=A\Z^n$ is a full-rank lattice, and its length spectrum is a discrete set. Fix any $x\in \Z^n$ and consider the triangle inequality $\big|\|A_kx\|-\|Ax\|\big|\leq \|(A_k-A)x\|$. For sufficiently large $k$, $\|A_kx\|=\|Ax\|$ since the spectra are discrete and identical for all $k$, and $A_k \to A$. There is now a $k$ big enough such that for all $x=e_i+e_j$, where $1\le i,j\le n$, we have
\[x^T\left(A_k^TA_k-A^TA\right)x =0.\]
Consequently, $A_k^TA_k=A^TA$ for all $k$ sufficiently large and thus $A_k=C_k A$ for $C_k \in \oO_n(\R)$.
\end{proof}

When studying limits of  lattices, Mahler's Compactness Theorem is crucial, without which this survey would not be complete.  According to Cassels, this theorem ``may be said to have completely transformed the subject'' in the context of lattice theory. See \cite[p. 136-139]{cassels2012introduction} for this quote and the proof of Mahler's Compactness Theorem.

\begin{theorem}[Mahler's Compactness Theorem] Let $\Lambda_i$ be an infinite sequence of lattices of the same dimension, satisfying the following two conditions: 
\begin{enumerate} 
\item  there exists a number $K>0$ such that $\vol(\Lambda_i)\leq K$ for all $i$;  
\item there exists a number $r>0$ such that $\inf_{0\neq v\in \Lambda_i}\|v\|\geq r$ for all $i$.
\end{enumerate} 
There is then a subsequence $\Lambda_{i_k}$ that converges to some lattice $\Lambda$.
\end{theorem}

As an application of the results collected above, one can prove the following Finiteness Theorem which was first demonstrated by Kneser in an unpublished work.  

\begin{theorem}[Finiteness Theorem, \cite{wolpert}]\label{th:fin} The number of non-isometric flat tori with a given Laplace spectrum is finite.
\end{theorem} 

The following proposition generalizes Wolpert's result to tuples.

\begin{prop}\label{prop: int tuple} If $\flat_n\ge k$, then there exist $k$ isospectral non-isometric even $n$-dimensional quadratic forms.
\end{prop}

\begin{proof} We follow the ideas of Wolpert. Consider the $n$-variable positive definite forms $Q_1,\ldots,Q_k$ that are isospectral and non-isometric. Then there are bijections $\phi_i:\ZZ^n\to \ZZ^n$ such that 
$$Q_1(\phi_1(z))=\cdots=Q_k(\phi_k(z))$$
for all $z\in \ZZ^n$. In particular, this tuple $Q_i$ lies in the set 
\[U:=\bigcap_{z\in \ZZ^n}\left\{(P_1,\ldots,P_k)\in \left(\mathcal{S}^n\right)^k:P_1(\phi_1(z))=\cdots=P_k(\phi_k(z)) \right\},\]
where $\mathcal{S}^n$ is the set of real symmetric $n\times n$ matrices, and it is not hard to see that $U$ is a linear space. Since it is defined by integer constraints, we can find a basis $(f_1^{(i)},\ldots,f_k^{(i)})$ of integer matrices for $U$. This means that $Q_1$ can be written as a linear combination of the $f_1^{(i)}$ with real coefficients $\lambda_i$. Since the set of positive definite forms is open, we can approximate $Q_1$ by a rational matrix $\tilde{Q_1}$, by choosing rational $\tilde{\lambda_i}\approx \lambda_i$, such that it is still positive definite and lies in the first factor of $U$. Then 
$$(\tilde{Q_1},\ldots,\tilde{Q_k}):=\sum_i\tilde{\lambda_i}(f_1^{(i)},\ldots,f_k^{(i)}),$$
is a tuple of isospectral rational positive definite forms in $U$, and up to a constant $\tilde{Q_i}$ are integral.
\end{proof}

In 1984, Suwa-Bier, a student of Kneser, obtained the following impressive result.

\begin{theorem}[Suwa-Bier \cite{suwa1984positiv}] The choir numbers are finite.
\end{theorem}
With this in mind, it is natural to investigate the asymptotic behavior of the choir numbers. We have already seen in Lemma \ref{le:schiemlemma} that the choir numbers are non-decreasing. 
In light of the four dimensional examples of isospectral, non-isometric lattices we obtain the following lower bound that we have not been able to find in the literature.

\begin{theorem} If $\flat_m\ge k$, then 
$$\flat_{n}\geq {\lfloor n/m \rfloor +k-1\choose k-1}.  $$
In particular, $\flat_n\geq \lfloor n/4\rfloor +1$, and $\flat_n$ tends towards infinity.  
\end{theorem}

\begin{proof} In dimension $mn$, for some positive integer $n$, we construct $n +k-1\choose k-1$ pairwise isospectral non-isometric flat tori. This would prove the statement by Lemma \ref{le:schiemlemma}. In dimension $m$ we have $k$ flat tori $\mathbb{T}_{\G_1},\ldots,\mathbb{T}_{\G_k}$ that are isospectral and non-isometric. Consider the sequence of $mn$-dimensional lattices $$\Omega_{i_1,\ldots,i_k}:=\G_1^{i_1}\times\cdots\times \G_k^{i_k},$$
with non-negative $i_j$ such that $i_1+\cdots+i_k=n$. There are $n +k-1\choose k-1$ different choices of the sequence $i_1,\ldots,i_k$. As a direct consequence of the Theorem \ref{th:inheritance1} and Proposition \ref{prop:cong1}, the flat tori $\mathbb{T}_{\Omega_{i_1,\ldots,i_k}}$ all share a common Laplace spectrum, but are pairwise non-isometric.
\end{proof}

The exact values of the choir numbers remains an open problem. Do they have polynomial or exponential growth?  It appears possible to use Suwa-Bier's techniques to obtain an upper bound for $\flat_n$.  This could be an interesting bachelor's or master's thesis.  Although it may be difficult to determine the exact expression for $\flat_n$ as a function of $n$, perhaps the following question is more tractable.
\begin{question} What is the asymptotic behavior of $\flat_n$ as $n \to \infty$? 
\end{question} 

\subsection{The fourth choir number} \label{ss: howmany} 
Since the third choir number has already been determined, the next step is to consider the fourth one. We could start by trying to show that $\flat_4=2$, by looking at triplets of quadratic forms instead of pairs. Let $\choir_n$ be a set containing a unique representative of each $n$-dimensional positive definite form.  We believe it is possible, albeit tedious, in the spirit of the Eisenstein and Schiemann reductions to obtain $\choir_4$ with this property.  Let $k,n $ be positive integers. We define
\[\duets_{n,k}:=\{(f_1,\ldots,f_k)\in (\choir_n)^k: f_i \textrm{ are isospectral for all } i=1, \ldots, k\}, \] 
so that for instance $\duets_{4,3}$ consists of triplets $(f,g,h)\in \choir_4\times \choir_4\times \choir_4$. It's possible to modify Schiemann's algorithm to deal with triplets instead of pairs. Our termination criterion in this case would be that for a polyhedral cone $T$, each triplet $(f,g,h)\in T$ should have either $f=g,f=h$ or $g=h$, since if the elements of a covering would have this property, then there are no triplets of different forms that all share representation numbers. The first main difficulty in working with triplets is that each iteration would go through three $\MIN$ sets instead of two. For this reason, the algorithm might be too slow, but considering the massive increase in modern computation power since Schiemann's thesis approximately thirty years ago, this might not pose too much difficulty.  However, it could be that the algorithm never terminates.  In that case, we would propose a computer search to obtain a triplet of isospectral non-isometric flat tori and perform the analogous algorithm with $\duets_{4,4}$ to try to prove $\flat_4=3$.  This process could be repeated iteratively.  In this way, using $\duets_{n,k}$, we could write a theoretically functioning program in the spirit of Schiemann's symphony for determining $\flat_n$ for each $n$.  The first step would be to answer 
\begin{question} What is the precise value of $\flat_4$?
\end{question} 


\subsection{The classification of quadratic forms}
The classification of even quadratic forms could be applied to shed light on the spectral geometry of flat tori, for example via Proposition \ref{prop: int tuple}.  One such classification of even quadratic forms is through \em genera. \em  Two even quadratic forms are of the same genus if they are $p$-adically equivalent for each $p$. One strong connection between genera and spectra is that the genus is determined by the spectrum precisely when the dimension is less than or equal to 4 \cite[p. 114]{conway}.  The classification of \em regular \em quadratic forms is related to Schiemann's work; in 3 dimensions Schiemann together with co-authors made a significant contribution \cite{jagy}.   A positive definite integral quadratic form $q$ of $k$ variables is \em regular \em if, for every positive integer $a$ for which the congruence $f(x) \equiv a$ mod $n$ has a solution for each positive integer $n$, the equation $f(x)=a$ has a solution $x \in \Z^k$.  This definition is nearly a century old, originally due to Dickson \cite{dickson}.  A quarter century later, Watson proved  \cite{watson_reg} that there are finitely many inequivalent primitive regular ternary quadratic forms. In 1997, Jagy, Kaplansky, and Schiemann \cite{jagy} produced a list containing representatives of all the \em possible \em equivalence classes of such forms; there are 913.  They proved that 891 of the list are indeed regular, leaving the remaining 22 as an open problem.  In 2011, Oh \cite{oh} proved 8 of the remaining 22 are regular.  In 2020, Oh and Kim classified all 49 regular ternary triangular forms; we refer to \cite{kim_oh} for the details.   These types of problems are generally approached with number theoretic techniques and tools, but it may be interesting to approach them from an analytic or geometric perspective.  What are the spectral and geometric implications for the associated flat tori?  What are the geometric features for the associated lattices?  

\subsection{The geometry of lattices and $k$-spectra} 
Our third perspective on the isospectral problem for flat tori concerns the geometry of lattices.  The $k$-spectrum is a geometric invariant of a lattice for which many questions are open.

\begin{defn}[$k$-spectrum] Let $\G$ be a lattice and $\Lambda\subseteq \G$ a sublattice. Write $[\Lambda]$ for the equivalence class of sublattices in $\G$ with respect to congruence. We define the \em $k$-spectrum \em of $\G$ to be the set
\[\mathcal{L}^k(\G):=\left\{\left([\Lambda],m_{\Lambda}\right): \Lambda \textnormal{ is a }k\textnormal{-rank sublattice of }\G \right\}. \] 
Here, $m_\Lambda=\#[\Lambda]$ is the number of sublattices of $\G$ that are congruent to $\Lambda$.
\end{defn}

Two full-rank lattices in $\R^n$ are congruent if and only if their $n$-spectra agree. Two lattices are isospectral if and only if their $1$-spectra agree. The $k$-spectrum is in this sense a generalization of the length and Laplace spectra.  It gives rise to new problems, perhaps the most natural being: 

\begin{question} For which triplets of positive integers $(n,k,m)$ with $n\ge k,m$, does $\mathcal{L}^k(\G_1)=\mathcal{L}^k(\G_2)$ imply $\mathcal{L}^m(\G_1)=\mathcal{L}^m(\G_2)$ for any $n$-rank lattices $\G_1$, $\G_2$? 
\end{question} 

It follows from our work here that the answer is positive for some triplets, for example $(3,1,3)$ as in \S \ref{s:schiemann}, and negative for others, like $(4,1,4)$ as in \S \ref{s:differentdonuts}. In the unpublished report \cite{claesspectral}, Claes showed that that the answer is yes if $n=m=3$, and $k=2$.  We show below that in any dimension $n$, the $k$ spectra determines the $m$ spectra for $m \leq k$.

\begin{prop} Let $n \geq k \geq m \geq 1$, and let $\G_i$ be full-rank lattices in $\R^n$. If the $k$-spectra of $\G_1$ and $\G_2$ agree, then their $m$-spectra also agree. 
\end{prop}

\begin{proof} The $k$-spectra agree if and only if there is a bijection
$$\phi: \G_1^k\to \G_2^k,$$
such that $\phi$ maps any set of linearly independent vectors $u_1,\ldots,u_k\in \G_1$ to a set of linearly independent $v_1,\ldots,v_k\in \G_2$, where the parallelotope spanned by $u_i$ is congruent to that which is spanned by $v_i$, meaning that we can order $v_i$ such that $(\langle u_i,u_j\rangle)_{ij}=(\langle v_i,v_j\rangle)_{ij}$.  Now consider the function $\phi':\G_1^{k-1}\to \G_2^{k-1}$, sending the sets $u_1,\ldots,u_{k-1}\in \G_1$ to $v_1,\ldots,v_{k-1}\in \G_2$ corresponding to $\phi$. Since the upper left $(k-1)\times (k-1)$ submatrices of the above $k\times k$ Gram matrices are equal, $\phi'$ satisfies precisely the condition that the $(k-1)$-spectra of $\G_1$ and $\G_2$ agree.
\end{proof}

As a consequence of the preceding proposition, the number of $n$-dimensional non-congruent lattices whose $k$-spectra all agree is bounded above by $\flat_n$, for any $k$. One could also phrase the $k$-spectrum as a property of flat tori or quadratic forms. In terms of quadratic forms, there is a connection to \textit{Siegel modular forms}; see \cite[Chapter 1]{pitale2019siegel}. For a certificate similar to Corollary \ref{cor: modisos} in terms of $k$-spectra, see \cite{richter2015sturm}.  We hope that readers investigating these and related problems will keep all three perspectives in mind and thereby reap the benefits of techniques from analysis, number theory, and geometry.

\section*{About the authors} 
Erik Nilsson is a PhD student in the department of Numerical Analysis at the Royal Institute of Technology in Stockholm, Sweden. His supervisor is Sara Zahedi. Erik completed his bachelor's, master's and civil engineering degrees in mathematics at Chalmers University, with one year of the master's studies completed at Politecnico di Milano.  He is interested in partial differential equations and their relation to science, industry, and nature. Erik's focus is currently aimed at finding mass conservative numerical schemes for equations governing flow phenomena in fractured domains.

Julie Rowlett is an associate professor in the Division of Analysis and Probability Theory at the joint Mathematical Sciences Department of Chalmers University and the University of Gothenburg in Sweden. She is also the director of the engineering mathematics program at Chalmers and a member of the Swedish National Committee for Mathematics. She completed her Habilitation at the Georg-August-Universit\"at G\"ottingen, PhD at Stanford University, and bachelor of science at the University of Washington. Julie's research interests include geometric, functional, and microlocal analysis; differential geometry; complex analysis and geometry; spectral theory; mathematical physics; dynamical systems; game theory; and interdisciplinary collaboration.

Felix Rydell is a PhD student in the department of Data Science and Artificial Intelligence at the Royal Institute of Technology in Stockholm, Sweden.  His supervisors are Kathl\'en Kohn and Fredrik Viklund.  Felix completed bachelor's and master's degrees in pure mathematics at the University of Gothenburg.  His research investigates and builds connections between algebraic geometry and machine learning together with artificial intelligence.  Felix is currently focused on different mathematical and geometrical approaches to computer vision.


\begin{thebibliography}{100}

\bibitem{berger1971spectre}
M.~Berger, P.~Gauduchon, and E.~Mazet, \emph{Le spectre d'une vari{\'e}t{\'e}
  riemannienne}, Le Spectre d’une Vari{\'e}t{\'e} Riemannienne, Springer,
  1971, pp.~141--241.

\bibitem{bezanson2012julia}
J.~Bezanson, S.~Karpinski, V.~B Shah, and A.~Edelman, \emph{Julia: A fast
  dynamic language for technical computing}, arXiv preprint arXiv:1209.5145
  (2012).

\bibitem{borrelli_pnas}
V.~Borrelli, F.~Lazarus, S.~Jabrane, and B.~Thibert, \emph{Flat tori in
  three-dimensional space and convex integration}, Proc. Natl. Acad. Sci. USA
  \textbf{109} (2012), no.~19, 7218--7223, DOI 10.1073/pnas.1118478109.
  \MR{2935570}

\bibitem{borrelli_book}
V.~Borrelli, F.~Lazarus, S.~Jabrane, and B.~Thibert,   Ensaios Matem\'{a}ticos [Mathematical Surveys], vol.~24, Sociedade Brasileira
  de Matem\'{a}tica, Rio de Janeiro, 2013. \MR{3154503}

\bibitem{bruning2012nodal}
J.~Br{\"u}ning and D.~Fajman, \emph{On the nodal count for flat tori},
  Communications in Mathematical Physics \textbf{313} (2012), no.~3, 791--813.

\bibitem{cassels2008rational}
J.~W.~S. Cassels, \emph{Rational quadratic forms}, Courier Dover Publications,
  2008.

\bibitem{cassels2012introduction}
J.~W.~S. Cassels, \emph{An introduction to the geometry of numbers}, Springer Science \&
  Business Media, 2012.

\bibitem{claesspectral}
J.~Claes, \emph{Spectral rigidity on $\mathbf{T}^n$},
  \url{http://www.math.uchicago.edu/\~may/VIGRE/VIGRE2011/REUPapers/Claes.pdf},
  2021.

\bibitem{cohn2014formal}
H.~Cohn, A.~Kumar, C.~Reiher, and A.~Sch{\"u}rmann, \emph{Formal duality and
  generalizations of the poisson summation formula}, Discrete geometry and
  algebraic combinatorics \textbf{625} (2014), 123--140.

\bibitem{conway}
J.~H. Conway, \emph{The sensual (quadratic) form}, Carus Mathematical
  Monographs, vol.~26, Mathematical Association of America, Washington, DC,
  1997, With the assistance of Francis Y. C. Fung. 

\bibitem{conwaysloane}
J.~H. Conway and N.~J.~A. Sloane, \emph{{Four-dimensional lattices with the
  same theta series}}, International Mathematics Research Notices \textbf{4}
  (1992), 93--96.

\bibitem{conway2013sphere}
J.~H. Conway and N.~J.~A. Sloane, \emph{Sphere packings, lattices and groups}, vol. 290, Springer
  Science \& Business Media, 2013.

\bibitem{cour-hil}
R.~Courant and D.~Hilbert, \emph{Methods of mathematical physics. {V}ol. {II}},
  Wiley Classics Library, John Wiley \& Sons, Inc., New York, 1989, Partial
  differential equations, Reprint of the 1962 original, A Wiley-Interscience
  Publication. 

\bibitem{dickson} 
L.~E.~Dickson, \em Ternary quadratic forms and congruences, \em Ann. of Math. (2), 28, no. 1--4, (1926/27), 333--341. 

\bibitem{earnest1991theta}
A.~G. Earnest and G.~L. Nipp, \emph{On the theta series of positive quaternary
  quadratic forms}, CR Math. Rep. Acad. Sci. Canada \textbf{13} (1991), 33--38.

\bibitem{ebeling2013lattices}
W.~Ebeling, \emph{Lattices and codes}, Lattices and Codes, Springer, 2013,
  pp.~1--32.

\bibitem{eichler1963einfuhrung}
M.~Eichler, \emph{Einf{\"u}hrung in die theorie der algebraischen zahlen und
  funktionen}, vol.~27, Springer, 1963.

\bibitem{engel2004lattice}
P.~Engel, L.~Michel, and M.~Senechal, \emph{Lattice geometry}, Tech. report,
  2004.

\bibitem{folland_fourier}
G.~Folland, \emph{Fourier analysis and its applications}, The Wadsworth \&
  Brooks/Cole Mathematics Series, Wadsworth \& Brooks/Cole Advanced Books \&
  Software, Pacific Grove, CA, 1992.

\bibitem{gawrilow2000polymake}
E.~Gawrilow and M.~Joswig, \emph{Polymake: a framework for analyzing convex
  polytopes}, Polytopes—combinatorics and computation, Springer, 2000,
  pp.~43--73.

\bibitem{gww}
C.~Gordon, D.~Webb, and S.~Wolpert, \emph{One cannot hear the shape of a drum},
  Bull. Amer. Math. Soc. \textbf{27} (1992), 134--138.

\bibitem{gosset}
T.~Gosset, \emph{On the regular and semi-regular figures in space of n
  dimensions}, Messenger of Mathematics \textbf{29} (1900), 43--48.

\bibitem{grunbaum1967convex}
B.~Gr{\"u}nbaum, V.~Klee, M.~A Perles, and G.~C Shephard, \emph{Convex
  polytopes}, vol.~16, Springer, 1967.

\bibitem{hall2015lie}
B.~Hall, \emph{Lie groups, lie algebras, and representations: an elementary
  introduction}, vol. 222, Springer, 2015.

\bibitem{heckemathematische}
E.~Hecke, \emph{Mathematische werke. 1959}, Vandenhocck and Ruprecht.

\bibitem{hein}
G.~Hein and J.M. Cervino, \emph{{The Conway-Sloane tetralattice pairs are
  non-isometric}}, Advances in Mathematics \textbf{228} (2011), 153--166.

\bibitem{hsia} 
J.~S.~Hsia, \em Regular positive ternary quadratic forms, \em 
Mathematika 28, no. 2, (1981), no. 2, 231--238. 

\bibitem{jagy}
W.~C.~ Jagy, I.~ Kaplansky, A.~Schiemann, \em There are 913 regular ternary forms, \em 
Mathematika 44, no. 2, (1997), 332--341. 

\bibitem{joswig2013polyhedral}
M.~Joswig and T.~Theobald, \emph{Polyhedral and algebraic methods in
  computational geometry}, Springer Science \& Business Media, 2013.

\bibitem{kac}
M.~Kac, \emph{Can one hear the shape of a drum?}, Amer. Math. Monthly
  \textbf{73} (1966), no.~4, part II, 1--23, DOI 10.2307/2313748. 

\bibitem{kacroot}
V.~Kac, \emph{Lecture 16 - root systems and root lattices},
  \url{https://math.mit.edu/classes/18.745/Notes/Lecture_16_Notes.pdf}, 2021.

\bibitem{kim_oh} 
M.~H.~Kim \& B.~K.~Oh, \em Regular ternary triangular forms, \em J. Number Theory, 214 (2020), 137--169.

\bibitem{kitaoka}
Y.~Kitaoka, \emph{{Positive definite quadratic forms with same representation
  numbers}}, Archiv der Mathematik \textbf{28} (1977), 495--497.

\bibitem{kneser1954theorie}
M.~Kneser, \emph{Zur theorie der kristallgitter.}, Mathematische Annalen
  \textbf{127} (1954), 105--106.

\bibitem{kneser}
M.~Kneser, \emph{{Lineare Relationen zwischen Darstellungszahlen quadratischer
  Formen}}, Mathematische Annalen \textbf{168} (1967), 31--39.

\bibitem{koblitz2012introduction}
N.~I. Koblitz, \emph{Introduction to elliptic curves and modular forms},
  vol.~97, Springer Science \& Business Media, 2012.

\bibitem{kuiper}
N.~H. Kuiper, \emph{On {$C^1$}-isometric imbeddings. {I}, {II}}, Nederl. Akad.
  Wetensch. Proc. Ser. A. {\bf 58} = Indag. Math. \textbf{17} (1955), 545--556,
  683--689. 

\bibitem{milnor}
J.~Milnor, \emph{Eigenvalues of the {L}aplace operator on certain manifolds},
  Proc. Nat. Acad. Sci. U.S.A. \textbf{51} (1964), 542, DOI
  10.1073/pnas.51.4.542. 

\bibitem{miyake2006modular}
T.~Miyake, \emph{Modular forms}, Springer Science \& Business Media, 2006.

\bibitem{nash_ii}
J.~Nash, \emph{{$C^1$} isometric imbeddings}, Ann. of Math. (2) \textbf{60}
  (1954), 383--396, DOI 10.2307/1969840.  

\bibitem{nebe2006self}
G.~Nebe, E.~M. Rains, and N.~J.~A. Sloane, \emph{Self-dual codes and invariant
  theory}, vol.~17, Springer, 2006.

\bibitem{neukirch2013algebraic}
J.~Neukirch, \emph{Algebraic number theory}, vol. 322, Springer Science \&
  Business Media, 2013.

\bibitem{oh}
B.~K.~Oh, \em Regular positive ternary quadratic forms, \em 
Acta Arith. 147, no. 3, (2011), 233--243. 


\bibitem{oishi2020positive}
R.~Oishi-Tomiyasu, \emph{Positive-definite ternary quadratic forms with the
  same representations over z}, International Journal of Number Theory
  \textbf{16} (2020), no.~7, 1493--1534.

\bibitem{o1980indecomposable}
OTO O'Meara, \emph{On indecomposable quadratic forms.}, Journal f{\"u}r die
  reine und angewandte Mathematik \textbf{1980} (1980), no.~317, 120--156.

\bibitem{pitale2019siegel}
A.~Pitale, \emph{Siegel modular forms: A classical and representation-theoretic
  approach}, vol. 2240, Springer, 2019.

\bibitem{richter2015sturm}
O.~K. Richter and M.~Westerholt-Raum, \emph{Sturm bounds for siegel modular
  forms}, Research in Number Theory \textbf{1} (2015), no.~1, 1--8.

\bibitem{schiemann1}
A.~Schiemann, \emph{{Ein Beispiel positiv definiter quadratischer Formen der
  Dimension 4 mit gleichen Darstellungszahlen}}, Archiv der Mathematik
  \textbf{54} (1990), 372--375.

\bibitem{schiemann1994ternare}
A.~Schiemann, \emph{Tern{\"a}re positiv definite quadratische formen mit gleichen
  darstellungszahlen: Dissertation zur erlangung des doktorgrades}, no. 268,
  Mathematischen Institut der Universit{\"a}t, 1994.

\bibitem{schiemann2}
A.~Schiemann, \emph{{Ternary positive definite quadratic forms are determined by
  their theta series}}, Mathematische Annalen \textbf{308} (1997), 507--517.

\bibitem{schurmann2009computational}
A.~Sch{\"u}rmann, \emph{Computational geometry of positive definite quadratic
  forms}, University Lecture Series \textbf{49} (2009).

\bibitem{serre2000matrices}
D.~Serre, \emph{Matrices: Theory and applications. 2002}, Graduate texts in
  mathematics (2000).

\bibitem{shiota1991theta}
K.~Shiota, \emph{{On theta series and the splitting of $ S\_2 (\backslash
  \Gamma\_0 (q)) $}}, Journal of Mathematics of Kyoto University \textbf{31}
  (1991), no.~4, 909--930.
  
  \bibitem{sturm_bound} 
  J.~Sturm, \em On the congruence of modular forms, \em Lecture Notes in Mathematics 1240, Springer (1987), 275--280. 

\bibitem{sunada} 
T.~Sunada, \em Riemannian Coverings and Isospectral Manifolds, \em vol. 121, no. 1, (1985), 169--186.  

\bibitem{suwa1984positiv}
Y.~Suwa-Bier, \emph{Positiv definite quadratische formen mit gleichen
  darstellungsanzahlen}, na, 1984.

\bibitem{tammela1977minkowski}
P.~Tammela, \emph{Minkowski reduction region for positive quadratic forms in
  seven variables}, Zapiski Nauchnykh Seminarov POMI \textbf{67} (1977),
  108--143.

\bibitem{terras2012harmonic}
A.~Terras, \emph{Harmonic analysis on symmetric spaces and applications ii},
  Springer Science \& Business Media, 2012.

\bibitem{van1956reduktionstheorie}
B.~L. van~der Waerden, \emph{Die reduktionstheorie der positiven quadratischen
  formen}, Acta Mathematica \textbf{96} (1956), no.~1, 265--309.
  
  \bibitem{watson_reg} 
  G.~L.~Watson, \em Some problems in the theory of numbers, \em Ph.D. thesis, Univ. London, (1953). 
  
  \bibitem{weyl} 
H.~ Weyl, \em \"Uber die asymptotische Verteilung der Eigenwerte, \em Nachrichten von der Gesellschaft der Wissenschaften zu
  G\"ottingen, Mathematisch-Physikalische Klasse, (1911), 110--117.

\bibitem{witt}
E.~Witt, \emph{Eine {I}dentit\"{a}t zwischen {M}odulformen zweiten {G}rades},
  Abh. Math. Sem. Hansischen Univ. \textbf{14} (1941), 323--337, DOI
  10.1007/BF02940750.  
  
  \bibitem{witt_cancel} 
  E.~Witt, \emph{Theorie der quadratischen Formen in beliebigen K\"orpern,} 
  J. Reine Angew. Math., 176, (1937), 31--44.  

\bibitem{wolpert}
S.~Wolpert, \emph{{The eigenvalue spectrum as moduli for flat tori}},
  Transactions of the American Mathematical Society \textbf{244} (1978),
  313--321, DOI 10.1090/s0002-9947-1978-0514879-9.

\end{thebibliography}
\end{document}